\documentclass[notitlepage,reqno,11pt]{amsart}
\usepackage{latexsym,amssymb, epsfig, amsmath,amsfonts, subfigure,amsthm}

\usepackage{rotating}
\usepackage[toc,page]{appendix}
\usepackage{color}
\usepackage{multirow}
\usepackage{relsize}
\usepackage{microtype}
\usepackage[foot]{amsaddr}

\usepackage[colorlinks, allcolors=blue]{hyperref}

\usepackage[margin=1in]{geometry}

\numberwithin{equation}{section}
 \newtheorem{assumption}{Assumption}[section]
\newtheorem{lemma}{Lemma}[section]
\newtheorem{theorem}{Theorem}[section]

\newtheorem{coro}{Corollary}[section]

\newtheorem{remark}{Remark}[section]

\newlength{\defbaselineskip}
\newcommand{\setlinespacing}[1]%
           {\setlength{\baselineskip}{#1 \defbaselineskip}}

\newcommand{\RR}{{\mathbb R}}

\newcommand{\NN}{{\mathbb N}}

\def\E{\mathbb{E}}
\def\P{\mathbb{P}}
\def\R{\mathbb{R}}

\newcommand{\sF}{{\mathcal{F}}}

\newcommand{\sL}{{\mathcal{L}}}

\newcommand{\beql}[1]{\begin{equation}\label{#1}}
\newcommand{\eeq}{\end{equation}}

\newcommand{\beqal}[1]{\begin{eqnarray}\label{#1}}
\newcommand{\eeqa}{\end{eqnarray}}
\newcommand{\beq}{\begin{displaymath}}
\newcommand{\eeqno}{\end{displaymath}}
\newcommand{\bali}[1]{\begin{align}\label{#1}}
\newcommand{\eali}{\begin{align}}
\newcommand{\balino}{\begin{align*}}
\newcommand{\ealino}{\begin{align*}}

\newcommand{\ep}{\epsilon}

\newcommand{\mfI}{\mathfrak{I}}
\newcommand{\mfi}{\mathfrak{i}}

\newcommand{\mfF}{\mathfrak{F}}
\newcommand{\mfa}{\mathfrak{a}}

\newcommand{\bD}{{\mathbf D}}

\newcommand{\bC}{{\mathbf C}}

\newcommand{\bone}{{\mathbf 1}}

\newcommand{\qasq}{\quad\mbox{as}\quad}
\newcommand{\qinq}{\quad\mbox{in}\quad}

\newcommand{\non}{\nonumber}

\newcommand{\baa}{\begin{eqnarray*}}
\newcommand{\eaa}{\end{eqnarray*}}

\newcommand{\ttl}{\Large PDE model for multi-patch epidemic models with migration \\[5pt] and infection-age dependent infectivity }

\begin{document}

\title[]{\ttl}

\author[Guodong \ Pang]{Guodong Pang$^*$}
\address{$^*$Department of Computational Applied Mathematics and Operations Research,
George R. Brown School of Engineering,
Rice University,
Houston, TX 77005}
\email{gdpang@rice.edu}

\author[{\'E}tienne \ Pardoux]{{\'E}tienne Pardoux$^\dag$}
\address{$^\dag$Aix--Marseille Universit{\'e}, CNRS, I2M, Marseille, France}
\email{etienne.pardoux@univ-amu.fr}


\begin{abstract} 
We study a stochastic epidemic model with multiple patches (locations), where individuals in each patch are categorized into three compartments, Susceptible, Infected and Recovered/Removed, and may migrate from one patch to another in any of the compartments. 
Each individual is associated with a random infectivity function which dictates the force of infection depending upon the age of infection (elapsed time since infection). We prove a functional law of large number for the epidemic evolution dynamics including the aggregate infectivity process, the numbers of susceptible and recovered individuals as well as the number of infected individuals at each time that have been infected for a certain amount of time.  From the limits, we derive a PDE model for the density of the number of infected individuals with respect to the infection age, which is a system of linear PDE equations with a boundary condition that is determined by a set of  integral equations. 
\end{abstract}


\maketitle



\allowdisplaybreaks

\section{Introduction}

Multi-patch epidemic models have been used to study infectious disease dynamics in different geographic areas \cite{sattenspiel1995structured,allen2007asymptotic,xiao2014transmission,bichara2018multi,PP-2020b,FPP-2022}. Most of the literature concerns Markovian models and the associated ODEs. 
In \cite{PP-2020b}, the authors study a non-Markovian multi-patch model with general exposed and infectious distributions as well as Markovian migration among the patches. That work extends the study of the homogeneous stochastic epidemic models in  \cite{PP-2020}. However, both works assumed a constant infection rate. In \cite{FPP2020b}, a stochastic epidemic model is studied to take into account varying infectivity, 
capturing the varying viral load phenomenon during infection as observed in \cite{he2020temporal}. In fact, Kermack and McKendrick \cite{KM27} already proposed deterministic epidemic models to study varying infectivity, and the FLLN limit in  \cite{FPP2020b} coincides with the integral equations in \cite{KM27}.
By tracking the age of infection (elapsed time since infection)  in that model with varying infectivity, in \cite{PP-2021}, the authors have studied the process counting the number of individuals at each time that have been infected for less than a certain amount of time, and derived a PDE model for the density of that process with respect to the infection age. The PDE model is comparable with the well known PDE models introduced by Kermack and McKendrick \cite{KM32}. 
 This homogeneous model with varying infectivity in  \cite{FPP2020b} is extended to a multi-patch multi-type model in \cite{FPP-2022}, however, the processes do not take into account the infection ages. 
We also refer to  \cite{clemenccon2008stochastic} and \cite{foutel2020individual} for individual-based stochastic epidemic models with contact-tracing and the associated PDE models as large population limits.

 In the present paper, we extend the study of epidemic models with infection-age dependent infectivity in \cite{PP-2021} to multi-patch models, and derive the associated PDE models. 
Specifically, we consider an individual-based stochastic epidemic model with multiple patches, where each individual is associated with a random infectivity function of the same law, and can migrate from one patch to another in each of the infection stages (susceptible, infected or recovered). 
The evolution dynamics at each time is described by the total force of infection, the number of susceptible individuals, the number of infected individuals that have been infected for less than a certain amount of time, and the number of recovered individuals. We prove a functional law of large numbers (FLLN) for these processes (Theorem \ref{thm-FLLN}), where the limits are a set of Volterra-type integral equations. 
We then derive a PDE model (Theorem \ref{thm-PDE}) from the limit of the proportions of infected individuals tracking the infection ages distribution, together with the other limits. 
We show that the PDE model is characterized by a system of linear equations, with a boundary condition also given by a set of Volterra-type integral equations.  The PDE model is derived under the assumption that the distribution of the infectious duration is absolutely continuous; however, we also discuss the more general case in Remark \ref{GeneralF}.

Since the seminal work in \cite{KM32}, a few articles have used PDE models to describe epidemic dynamics with infection-age dependent infectivity. See, for example,  \cite{inaba2001kermack, inaba2004mathematical, ZhangPeng2007,magal2013two,BCF-2019} and references therein. They all use the  hazard rate function of the infection durations as a way to model the dependence upon the infection ages. For many scenarios, constructing the PDE models directly using the hazard rate functions is feasible, and sometimes, it is a very convenient method. However, for the multi-patch model with migration as we consider in this paper, it seems difficult to directly construct the PDE model using hazard rate functions to describe the dependence on the infection ages together with the migration dynamics.   
It is then important that we start with an individual-based stochastic model and then derive the PDE models as the scaling limits of the stochastic models. As a consequence, we find that the PDE model also uses hazard rate function of the infectious duration  (see the PDE equation in \eqref{eqn-PDE}).

To prove the FLLN, we employ the weak convergence criterion for stochastic processes taking values in the $\bD_\bD$ space, see Theorem \ref{thm-DD-conv0} (established in \cite{PP-2021}). The proof for the multi-patch model relies on an important  observation that the process tracking how long individuals have been infected has an integral representation (Lemma \ref{lem-mfI-rep} and \eqref{eqn-bar-mfI-N}). The convergence criterion is used for three components in the integral representation (Lemmas \ref{lem-barM-conv}, \ref{lem-bar-mfI0-conv} and \ref{lem-bar-mfI1-conv}), together with properties of stochastic integrals with respect to the associated Poisson random measures. 

\medskip

\paragraph{\bf Organization of the paper}
The rest of the paper is organized as follows. In Section \ref{sec-model}, we describe the model in detail, and state the FLLN result. In Section \ref{sec-PDE}, we state the PDE model and its derivation, and prove the existence and uniqueness of its solution. The proof for the FLLN is given in Section \ref{sec-proof-FLLN}. In the Appendix, we recall two results on the weak convergence of stochastic processes used in the proof.

\section{The Model and FLLN} \label{sec-model}

We consider a multi-patch epidemic model with infection-age dependent infectivity described as follows. Individuals in each patch belong to either of the susceptible, infected or recovered/removed compartments, and may migrate from one patch to another in any of the three compartments. Each individual is associated with a random infectivity function,  which depends on the age  of infection (elapsed time since infection). 

Let $N$ be the total population size and $L$ be the number of patches.
For each $\ell \in \sL:=\{1,\dots,L\}$, let $S^N_\ell(t)$, $I^N_\ell(t)$ and $R^N_\ell(t)$ denote the numbers of individuals in patch $\ell$ that are susceptible, infectious and recovered at time $t$, respectively. 
Then we have the balance equation:
$$
N\;=\; \sum_{\ell=1}^L \big(S^N_\ell(t) + I^N_\ell(t) + R^N_\ell(t)\big) \,, \quad t \ge 0\,.
$$
Assume that $S^N_\ell(0)>0$,  $\sum_{\ell=1}^L I^N_\ell(0)>0$ and $R^N_\ell(0)\ge 0$,  $\ell \in \sL$. 
Let $B^N_\ell(t) = S^N_\ell(t) + I^N_\ell(t) + R^N_\ell(t)$ for $t\ge 0$ and each $\ell$. 
In addition, let $\mfI^N_\ell(t,\mfa)$ be the number of infected individuals in patch $\ell$ at time $t$ that have been infected for less than or equal to $\mfa$. Note that the initially infected individuals may or may not have recovered by time $0$. If recovered, they will be counted in  $R^N_\ell(0)$ and otherwise, in $I^N_\ell (0)$, for some $\ell\in \sL$. The distribution of individuals in  $I^N_\ell (0)$ according to their infection ages is given by $\mfI^N_\ell(0,\mfa)$. 
Also, let $A^N_\ell(t)$ be the number of newly infected individuals in patch $\ell$ by time $t$ after time 0. 
 
For each individual $i$ that becomes infected in patch $\ell$, let $\lambda_{i}^{\ell}: \RR_+\to \RR_+$ be the associated random infectivity function. Similarly, for each individual $j$ that is infected in patch $\ell$ at time zero, let $\lambda_{j}^{0, \ell}(t)$ be the associated infectivity function. 
Assume that the random functions $\{\lambda_{i}^{\ell}(\cdot)\}_{i, \ell}$ and $\{\lambda_{j}^{0,\ell}(\cdot)\}_{j,\ell}$  are  independent and have the same law. This is reasonable since we model the same disease. 
We write $\lambda(t)$  as the generic random functions for these sequences. 
Associated with the infectivity functions, we let $\eta^{\ell}_{i} = \sup\{t>0: \lambda_{i}^{\ell} (t)>0\}$ for $i \in \NN$ and  $\eta^{\ell}_{j} = \sup\{t>0: \lambda_{j}^{0,\ell} (t)>0\}$ for $j=1,\dots, I^N_\ell(0)$. By the i.i.d. assumption on $\{\lambda_{i}^{\ell}(\cdot)\}_{i, \ell}$ and $\{\lambda_{j}^{0,\ell}(\cdot)\}_{j,\ell}$, the variables $\{\eta^{\ell}_{i}\}_{i,\ell}$ and $\{\eta^{\ell}_{j}\}_{j,\ell}$ are also i.i.d., and we let $F$ be the associated c.d.f. and denote $F^c=1-F$. 

Let $\{\tau_{i}^{\ell,N}, i \in \NN\}$ be the event times associated with the infection process $A^N_\ell(t)$. Assume that 
$0 < \tau_{1}^{\ell,N}< \tau_{2}^{\ell,N}<\cdots$ so that $A^N_\ell(t)= \max\{i\ge 1:  \tau_{i}^{\ell,N} \le t\}$ with $A^N_\ell(0)=0$.
For the initially infected individuals in patch $\ell$, we let $\{\tau_{j,0}^{\ell,N}, j =1,\dots, I^N_\ell(0)\}$ be the times at which the initially infected individuals at time 0 became infected. Note that we label the initially infected individuals by the patch where they are at time $0$, irrespective of where they have been infected. We do not follow the movements of the individuals before time $0$. 
Then $\tilde{\tau}_{j,0}^{\ell,N} = -\tau_{j,0}^{\ell,N}$, $ j =1,\dots, I_\ell^N(0)$, represent the amount of time that an initially infected individual has been infected by time $0$, that is, the age of infection at time $0$. 
WLOG, assume that $0 > \tau_{1,0}^{\ell,N}> \tau_{2,0}^{\ell,N}> \cdots >\tau_{I^N_\ell(0),0}^{\ell,N}$ (or equivalently 
$0 < \tilde{\tau}_{1,0}^{\ell,N}< \tilde{\tau}_{2,0}^{\ell,N}< \cdots < \tilde{\tau}_{I^N_\ell(0),0}^{\ell,N}$).  


Moreover, we assume that the infection times $\{\tau_{i}^{\ell,N}, i \in \NN\}$ are independent of the random infectivity functions $\{\lambda_{i}^{\ell}(\cdot)\}_{i, \ell}$, and similarly,   $\{\tau_{j,0}^{\ell,N}, j =1,\dots, I^N_\ell(0)\}$ are also independent of  $\{\lambda_{j}^{0,\ell}(\cdot)\}_{j,\ell}$. 
If $\eta^\ell_j\le \tilde{\tau}_{j,0}^{\ell,N}$, then the individual $j$ has recovered by time 0 and belongs to $R^N_\ell(0)$. On the other hand, if  $\eta^\ell_j > \tilde{\tau}_{j,0}^{\ell,N}$, then the individual $j$ is still infected at time $0$ and belongs to $I^N_\ell(0)$, and we let $\eta^{0,\ell}_{j} = \sup\big\{t>0: \lambda_{j}^{0,\ell} (\tilde{\tau}_{j,0}^{\ell,N}+t)>0\big\}$ be the  remaining infected period for the individual. The conditional distribution of $\eta^{0,\ell}_{j}$ given that $\eta^\ell_j > \tilde{\tau}_{j,0}^{\ell,N}=s>0$ is 
  \[
\P\big(\eta^{0,\ell}_{j} >t \,\big|\, \eta^{\ell}_j >\tilde{\tau}_{j,0}^{\ell,N}=s\big) = \P\big(\eta^{\ell}_{j}-\tilde{\tau}_{j,0}^{\ell,N} >t \,\big|\, \eta^{\ell}_j >\tilde{\tau}_{j,0}^{\ell,N}=s\big) = \frac{F^c(t+s)}{F^c(s)}, \quad \text{for}\quad t, s>0. 
\]
For notational convenience, we let $F_0(t|s):=1- \frac{F^c(t+s)}{F^c(s)}$. 
In addition, define $\mfI_{\ell}^N(0,\mfa) = \max\{1 \le j \le I^N_\ell(0): \tilde{\tau}_{j,0}^{\ell,N} \le \mfa\}$, which represents the number of initially infected individuals in patch $\ell$ that remain infected and have been infected for less or equal to $\mfa$ at time 0. Evidently, $\mfI_{\ell}^N(0,0)=0$ and $I^N_\ell(0) = \mfI^{N}_\ell(0, +\infty)$ a.s. for all $\ell\in \sL$.


Individuals may migrate from one patch to another in any of the S-I-R stages. Assume that 
the migration rates depend on the patch and the stage of the epidemic (S-I-R), that is, an individual in patch $\ell$ in stage $S$ (resp. stages I and R), migrates from patch $\ell$ to $\ell'$ with rates $\nu^{S}_{\ell, \ell'}$ (resp. $\nu^{I}_{\ell, \ell'}$ and $\nu^{R}_{\ell, \ell'}$). 
Note that for each $\ell \in \sL$, $\nu^{S}_{\ell, \ell}= - \sum_{\ell': \ell'\neq \ell} \nu^{S}_{\ell, \ell'}$, and similarly, for
$\nu^{I}_{\ell, \ell'}$ and $\nu^{R}_{\ell, \ell'}$.  
In order to follow the movement of the infected individuals, let $X_{i}^{\ell}(t)$ be the patch at time $t$ of the $i$-th newly infected individual that becomes infected in patch $\ell$.
 Then $X_{i}^{\ell}(t)$ is a Markov process associated with rates $\nu^{I}_{\ell, \ell'}$, $\ell' \in \sL$, and let $p_{\ell,\ell'}(t) = \P(X_{i}^{\ell}(t) = \ell'| X_{i}^{\ell}(0)= \ell)$ for $\ell, \ell' \in \sL$ and $t\ge 0$. 
 Let $X_{j}^{0,\ell}(t)$ be the patch at time $t$ of the $j$-th initially infected individual that was in patch $\ell$ at time 0. Assume that  $\{X_{j}^{0,\ell}(t)\}_{j, \ell}$ has the same transition probability functions $p_{\ell,\ell'}(\cdot) $ as $\{X_{i}^{\ell}(t)\}_{i, \ell}$ for each $\ell \in \sL$. 

The aggregate infectivity in patch $\ell$ at time $t$ is given by 
\begin{equation} \label{eqn-mfF}
\mfF^N_\ell(t) = \sum_{\ell'=1}^L \sum_{j=1}^{I^N_{\ell'}(0)} \lambda^{0,\ell'}_j(\tilde{\tau}_{j,0}^{\ell',N}+t) {\bf1}_{X^{0, \ell'}_j(t) = \ell} +  \sum_{\ell'=1}^L\sum_{i=1}^{A^N_{\ell'}(t)} \lambda_i^{\ell'}(t-\tau^{\ell',N}_{i}) {\bf1}_{X^{ \ell'}_i(t-\tau^{\ell',N}_{i}) = \ell}\,. 
\end{equation}

We consider the following instantaneous infection rate function: 
\begin{equation} \label{eqn-Upsilon} 
\Upsilon^N_\ell(t) = \frac{S_\ell^N(t) \sum_{\ell'=1}^L \beta_{\ell \ell'}\mfF^N_{\ell'}(t) }{N^{1-\gamma} (B_\ell^N(t))^\gamma} = \left(\frac{B_\ell^N(t)}{N}\right)^{1-\gamma} \frac{S_\ell^N(t)}{B_\ell^N(t)} \sum_{\ell'=1}^L \beta_{\ell \ell'}\mfF^N_{\ell'}(t)\,,
\end{equation} 
where the constants $\beta_{\ell \ell'}>0$ represent the impact from patch $\ell'$ upon patch $\ell$, and $\gamma \in [0,1]$.  Assume that $\beta_{\ell \ell'} \le \beta^*<\infty$. 
When $\gamma=0$,  an infected individual in patch $\ell'$ encounters a susceptible individual from patch $\ell$ with a rate $S_\ell^N(t)/N$, and 
when $\gamma=1$, that rate is equal to $\frac{S_\ell^N(t)}{B_\ell^N(t)}$. When $\gamma \in (0,1)$, the infection rate is a mix of the two extreme scenarios.   We also refer the readers to \cite{PP-2020b,FPP-2022} for further discussions on such infection rate functions. 
Then we can write the counting process of newly infected individuals at patch $\ell$ as 
\begin{equation} \label{eqn-An}
A^N_\ell(t) = \int_0^t \int_0^\infty {\bf1}_{u \le \Upsilon^N_\ell(s^-) } Q_\ell(ds, du),
\end{equation} 
where $Q_\ell$'s are mutually independent standard Poisson random measures on $\RR_+^2$. 
 
The number of susceptible individuals  in patch $\ell$ at each time $t$ can be represented by
\begin{equation} \label{eqn-Sn} 
S^N_\ell(t) = S^N_\ell(0) - A^N_\ell(t) - \sum_{\ell'=1}^L P^S_{\ell,\ell'}\left( \nu^{S}_{\ell, \ell'} \int_0^t S^N_{\ell}(s) ds\right)  + \sum_{\ell'=1}^L P^S_{\ell',\ell}\left( \nu^{S}_{\ell', \ell} \int_0^t S^N_{\ell'}(s) ds\right)\,,
\end{equation}
where $P^S_{\ell,\ell'}$'s are mutually independent rate-1 Poisson processes. 

The number of infected individuals in patch $\ell$ at each $t$ that have been infected for less than or equal to $\mfa$ can be represented by 
\begin{align} \label{eqn-mfI-N}
\mfI^N_\ell(t, \mfa)  &  = \sum_{\ell'=1}^L \sum_{j=1}^{I^N_{\ell'}(0)} {\bf1}_{\eta^{0,\ell'}_j >t} {\bf1}_{\tilde{\tau}_{j,0}^{\ell',N} \le (\mfa-t)^+} {\bf1}_{X_j^{0, \ell'}(t) = \ell} \non  \\
& \quad +  \sum_{\ell'=1}^L \sum_{i=A^N_{\ell'}((t-\mfa)^+)+1}^{A^N_{\ell'}(t)} {\bf1}_{\tau_{i}^{\ell',N}+\eta^{\ell'}_i >t}  {\bf1}_{X_i^{\ell'}(t-\tau_{i}^{\ell',N}) = \ell}\,. 
\end{align}
The number of infected individuals in patch $\ell$ at each $t$ is then equal to
\[
I^N_\ell(t) = \mfI^N_\ell(t, \infty).   
\]
The number of recovered/removed individuals  in patch $\ell$ at each time $t$ can be represented by
\begin{align} \label{eqn-Rn} 
R^N_\ell(t)
& =  R^N_\ell(0)+  \sum_{\ell'=1}^L \sum_{j=1}^{I^N_{\ell'}(0)} {\bf1}_{\eta^{0,\ell'}_j \le t}{\bf1}_{X_j^{0, \ell'}(\eta^{0,\ell'}_j) = \ell}   +  \sum_{\ell'=1}^L \sum_{i=1}^{A^N_{\ell'}(t)} {\bf1}_{\tau_{i}^{\ell',N}+\eta^{\ell'}_i \le t}  {\bf1}_{X_i^{\ell'}(\eta^{\ell'}_i) = \ell} \non \\
& \quad - \sum_{\ell'=1}^L P^R_{\ell,\ell'}\left( \nu^{R}_{\ell, \ell'} \int_0^t R^N_{\ell}(s) ds\right)  + \sum_{\ell'=1}^L P^R_{\ell',\ell}\left( \nu^{R}_{\ell', \ell} \int_0^t R^N_{\ell'}(s) ds\right)\,,
\end{align} 
where $P^R_{\ell,\ell'}$'s are mutually independent rate-1 Poisson processes, independent of 
$P^S_{\ell,\ell'}$'s. 

It is clear that the four processes $S^N_\ell, \mfF^N_\ell, \mfI^N_\ell, R^N_\ell$ describe the epidemic evolution dynamics of our model. 
We provide an alternative representation of $\mfI^N_\ell(t,\mfa)$ in the following lemma. 

\begin{lemma} \label{lem-mfI-rep} 
We have
\begin{align} \label{eqn-mfI-N-new} 
\mfI^N_\ell(t, \mfa) 
&  = \mfI^N_{\ell}(0,(\mfa-t)^+ ) - \sum_{\ell'=1}^L  \sum_{j=1}^{\mfI^N_{\ell'}(0,(\mfa-t)^+ )} {\bf1}_{\eta^{0,\ell'}_j \le t} {\bf1}_{X_j^{0, \ell'}(\eta^{0,\ell'}_j ) = \ell}  \non\\
& \quad   + A^N_{\ell}(t) - A^N_{\ell}((t-\mfa)^+)   - \sum_{\ell'=1}^L \sum_{i=A^N_{\ell'}((t-\mfa)^+)+1}^{A^N_{\ell'}(t)} {\bf1}_{\tau_{i}^{\ell',N}+\eta^{\ell'}_i \le t}  {\bf1}_{X_i^{\ell'}(\eta^{\ell'}_i) = \ell}  \non \\
& \quad  -   \sum_{\ell'=1}^L    \int_{(t-\mfa)^+}^t\int_0^\infty{\bf1}_{u\le \nu^{I}_{\ell, \ell'}\mfI^N_{\ell}(s, \mfa-(t-s))}Q^I_{\ell,\ell'}(ds,du)   \non \\
& \quad  +  \sum_{\ell'=1}^L   \int_{(t-\mfa)^+}^t\int_0^\infty{\bf1}_{u\le \nu^{I}_{\ell', \ell}\mfI^N_{\ell'}(s, \mfa- (t-s))}Q^I_{\ell',\ell}(ds,du) \,, 
\end{align}
where $Q^{I}_{\ell,\ell'}(ds,du)$, $\ell,\ell'\in \sL$,  are mutually independent standard PRMs on $\RR^2_+$, independent of $P^S_{\ell,\ell'}$ and $P^R_{\ell,\ell'}$, $\ell,\ell'\in \sL$. 
\end{lemma}

\begin{proof}
Recall the expression of $\mfI^N_\ell(t, \mfa) $ in \eqref{eqn-mfI-N}.
For the first term, 
we have in the summation over $\ell'$,  if $\ell'=\ell$, 
\begin{align*}
   \sum_{j=1}^{\mfI^N_{\ell}(0,(\mfa-t)^+ )} {\bf1}_{\eta^{0,\ell}_j >t} {\bf1}_{X_j^{0, \ell}(t) = \ell}  
&=\mfI^N_{\ell}(0,(\mfa-t)^+ ) -    \sum_{j=1}^{\mfI^N_{\ell}(0,(\mfa-t)^+ )} {\bf1}_{\eta^{0,\ell}_j \le t} {\bf1}_{X_j^{0, \ell}(\eta^{0,\ell}_j ) = \ell} - \sum_{\ell': \ell'\neq \ell} Y^{N,0}_{\ell, \ell'}(t,\mfa) \,,
\end{align*}
and  if $\ell'\neq \ell$,
\begin{align*}
\sum_{j=1}^{\mfI^N_{\ell'}(0,(\mfa-t)^+ )} {\bf1}_{\eta^{0,\ell'}_j >t} {\bf1}_{X_j^{0, \ell'}(t) = \ell}  
&=  Y^{N,0}_{\ell', \ell}(t,\mfa) - \sum_{j=1}^{\mfI^N_{\ell'}(0,(\mfa-t)^+ )} {\bf1}_{\eta^{0,\ell'}_j \le t} {\bf1}_{X_j^{0, \ell'}(\eta^{0,\ell'}_j ) = \ell}\,,
\end{align*}
where  $Y^{N,0}_{\ell', \ell}(t,\mfa)$ is the number of the initially infected individuals in patch $\ell'$ that are in patch $\ell$ at time $t \wedge \eta^{0,\ell'}_j$, for $j=1,\dots, \mfI^N_{\ell'}(0,(\mfa-t)^+ )$. 

Next, for the second term, 
we have in the summation, if $\ell'=\ell$, 
\begin{align*}
&  \sum_{i=A^N_{\ell}((t-\mfa)^+)+1}^{A^N_{\ell}(t)} {\bf1}_{\tau_{i}^{\ell,N}+\eta^{\ell}_i >t}  {\bf1}_{X_i^{\ell}(t-\tau_{i}^{\ell,N}) = \ell}\\
&  = A^N_{\ell}(t) - A^N_{\ell}((t-\mfa)^+) -  \sum_{i=A^N_{\ell}((t-\mfa)^+)+1}^{A^N_{\ell}(t)} {\bf1}_{\tau_{i}^{\ell,N}+\eta^{\ell}_i \le t}  {\bf1}_{X_i^{\ell}(\eta^{\ell}_i) = \ell} -  \sum_{\ell': \ell'\neq \ell} Y^{N}_{\ell, \ell'}(t,\mfa)\,,
\end{align*}
and if $\ell'\neq \ell$, 
\begin{align*}
&  \sum_{i=A^N_{\ell'}((t-\mfa)^+)+1}^{A^N_{\ell'}(t)} {\bf1}_{\tau_{i}^{\ell',N}+\eta^{\ell'}_i >t}  {\bf1}_{X_i^{\ell'}(t-\tau_{i}^{\ell',N}) = \ell}  = Y^{N}_{\ell', \ell}(t,\mfa) - \sum_{i=A^N_{\ell'}((t-\mfa)^+)+1}^{A^N_{\ell'}(t)} {\bf1}_{\tau_{i}^{\ell',N}+\eta^{\ell'}_i \le t}  {\bf1}_{X_i^{\ell'}(\eta^{\ell'}_i) = \ell}\,,
\end{align*}
where  $Y^{N}_{\ell', \ell}(t,\mfa)$ is the number of the newly infected individuals in patch $\ell'$  that are in patch $\ell$ at time $t \wedge (\tau_{i}^{\ell',N}+\eta^{\ell'}_i)$, for $i=A^N_{\ell'}((t-\mfa)^+)+1,\dots, A^N_{\ell'}(t)$. 

We then observe that
\begin{align*}
 &  \sum_{\ell'=1}^L  \left(Y^{N,0}_{\ell', \ell}(t,\mfa) + Y^{N}_{\ell', \ell}(t,\mfa) - Y^{N,0}_{\ell, \ell'}(t,\mfa) - Y^{N}_{\ell, \ell'}(t,\mfa)\right)  \\
 & =    \sum_{\ell'=1}^L \bigg(  \int_{(t-\mfa)^+}^t\int_0^\infty{\bf1}_{u\le \nu^{I}_{\ell', \ell}\mfI^N_{\ell'}(s, \mfa- (t-s))}Q^I_{\ell',\ell}(ds,du)-  \int_{(t-\mfa)^+}^t\int_0^\infty{\bf1}_{u\le \nu^{I}_{\ell, \ell'}\mfI^N_{\ell}(s, \mfa-(t-s))}Q^I_{\ell,\ell'}(ds,du) \bigg)\,. 
\end{align*} 
The interpretation of the identity is as follows. The left hand side counts the total number of individuals (initially and newly infected) that have migrated from all patches $\ell'$ into patch $\ell$, minus those out of patch $\ell$, but only the individuals with an infection age less than or equal to $(t-\mfa)^+$ at time $t$,  or recovered by time $t$. The right hand side represents the same counts by using the processes $\mfI^N_{\ell}(t, \mfa)$, but noting the $\mfI^N_{\ell}(s,\mfa-(t-s))$ inside the integral as the infection age evolves with $s$ changes from $(t-\mfa)^+$ to $t$.

Combining the above, we obtain the representation of $\mfI^N_\ell(t, \mfa) $ in the lemma. 
\end{proof} 

Throughout the paper, let $\bD=\bD(\RR_+;\RR)$ denote the space of $\RR$--valued c{\`a}dl{\`a}g functions defined on $\RR_+$.  Convergence in $\bD$ means convergence in the  Skorohod $J_1$ topology, see Chapter 3 of \cite{billingsley1999convergence}. 
 Also, $\bD^k$ stands for the $k$-fold product equipped with the product topology. 
 Let $\bC$ be the subset of $\bD$ consisting of continuous functions.   Let $\bD_\bD= \bD(\RR_+; \bD(\RR_+;\RR))$ be the $\bD$-valued $\bD$ space. In particular, the processes $\mfI_\ell^N(t,\mfa)$ have sample paths in $\bD_\bD$. 
 
 We define the LLN scaled processes $\bar{Z}^N= N^{-1} Z^N$ for any process $Z$. We first impose the following conditions on the initial quantities.

\begin{assumption}\label{AS-initial}
There exist deterministic continuous nondecreasing functions $\bar\mfI_\ell(0,\cdot)$ on $\R_+$ with $\bar\mfI_\ell(0,0)=0$ and constants  $\bar{S}_\ell(0), \bar{R}_\ell(0)\in [0,1]$,   $\ell \in \sL$,  such that 
\[(\bar{S}^N_\ell(0), \bar\mfI^N_\ell(0,\cdot), \bar{R}^N_\ell(0))_{  \ell \in \sL} \to (\bar{S}_\ell(0), \bar\mfI_\ell(0,\cdot), \bar{R}_\ell(0))_ {\ell \in \sL} \qinq \R_+^L \times \bD^L \times \R_+^L  \] 
 in probability as $N\to \infty$. Let $\bar{I}_\ell(0)= \bar\mfI_\ell(0,+\infty)$ for each $\ell \in \sL$. Then 
the convergence implies that $(\bar{I}^N_\ell(0), \ell \in\sL) \to (\bar{I}_\ell(0), \ell \in\sL) $ in $\R_+^L$  in probability as $N\to \infty$.  In addition, assume that $\sum_{\ell \in \sL} \bar{I}_\ell(0) >0$, $\sum_{\ell\in \sL}(\bar{S}_\ell(0)+ \bar{I}_\ell(0) + \bar{R}_\ell(0))=1$, and that the functions $\mfa\mapsto \bar\mfI_\ell(0,\mfa)$ satisfy the following assumptions: 
 there exist constants $C,\,\alpha>0$ such that $\bar\mfI_\ell(0,\mfa)-\bar\mfI(0,\mfa-\delta)\le C\delta^\alpha$ for all $1\le \ell\le L$, $\mfa>0$, $\delta>0$. 
\end{assumption}

We then impose the following conditions on the random infectivity functions. Recall that $\{\lambda_{j}^{0,\ell}(\cdot)\}_{j,\ell}$ and $\{\lambda_{i}^{\ell}(\cdot)\}_{i,\ell}$ have the same law.

\begin{assumption}\label{AS-lambda}
Let $\lambda(\cdot)$ be a process representing $\{\lambda_{j}^{0,\ell}(\cdot)\}_{j,\ell}$ and $\{\lambda_{i}^{\ell}(\cdot)\}_{i,\ell}$ with the same law. Assume that $\lambda(\cdot)\in \bD$, and  there exists a constant $\lambda^*$ such that  $\sup_{t \ge0} \lambda(t) \le \lambda^*$ a.s. 
Let $\bar{\lambda}(t) = \E[\lambda(t)]$ for $t\ge 0$.
\end{assumption} 

\begin{theorem} \label{thm-FLLN}
Under Assumptions \ref{AS-initial} and \ref{AS-lambda}, 
\begin{equation}
(\bar{S}^N_\ell, \bar{\mfF}^N_\ell, \bar{\mfI}^N_\ell, \bar{R}^N_\ell)_{\ell\in\sL} \to (\bar{S}_\ell, \bar{\mfF}_\ell, \bar{\mfI}_\ell, \bar{R}_\ell)_{\ell\in\sL}
\end{equation} 
in probability, locally uniformly in $t$ and $\mfa$ as $N\to\infty$, where the limits are the unique continuous solution to the following set of integral equations, for  $t, \mfa \ge 0$, 
\begin{align}
\bar{S}_\ell(t) &= \bar{S}_\ell(0) - \int_0^t \bar\Upsilon_\ell(s) ds + \sum_{\ell'=1}^L \int_0^t \Big( \nu^S_{\ell', \ell}\bar{S}_{\ell'}(s) - \nu^S_{\ell, \ell'}\bar{S}_{\ell}(s)  \Big) ds\,, \label{eqn-barS}\\
\bar{\mfF}_\ell(t) &= \sum_{\ell'=1}^L p_{\ell',\ell}(t) \int_0^{\infty} \bar{\lambda}(\mfa+t) \bar\mfI_{\ell'}(0,d\mfa) 
+  \sum_{\ell'=1}^L   \int_0^t \bar{\lambda}(t-s) p_{\ell',\ell}(t-s)  \bar\Upsilon_{\ell'}(s) ds\,, \label{eqn-bar-mfF}\\
\bar{\mfI}_\ell(t,\mfa) &= \bar\mfI_{\ell}(0,(\mfa-t)^+ )  -  \sum_{\ell'=1}^L  \int_0^{(\mfa-t)^+}  \bigg( \int_0^t p_{\ell',\ell}(u)    F_0(du|y) \bigg)    \bar\mfI_{\ell'}(0,dy)  \non \\
& \quad +\int_{(t-\mfa)^+}^t \bar\Upsilon_\ell(s) ds - \sum_{\ell'=1}^L \int_{(t-\mfa)^+}^t \int_0^{t-s} p_{\ell',\ell}(u) F(du)  \bar\Upsilon_{\ell'}(s) ds \label{eqn-bar-mfI} \\
& \quad + \sum_{\ell'=1}^L \int_{(t-\mfa)^+}^t \Big( \nu^I_{\ell', \ell} \bar\mfI_{\ell'}(s,\mfa-(t-s))  - \nu^I_{\ell, \ell'} \bar\mfI_{\ell}(s,\mfa-(t-s))   \Big) ds\,, \non \\
\bar{R}_\ell(t) &= \bar{R}_\ell(0) +  \sum_{\ell'=1}^L   \int_0^{\infty}  \bigg( \int_0^t p_{\ell',\ell}(u)    F_0(du|\mfa) \bigg)    \bar\mfI_{\ell'}(0,d\mfa) \non\\
& \quad + \sum_{\ell'=1}^L \int_{0}^t \int_0^{t-s} p_{\ell',\ell}(u) F(du)  \bar\Upsilon_{\ell'}(s) ds  + \sum_{\ell'=1}^L \int_0^t \Big( \nu^R_{\ell', \ell}\bar{R}_{\ell'}(s) - \nu^R_{\ell, \ell'}\bar{R}_{\ell}(s)  \Big) ds\,, \label{eqn-barR}
\end{align} 
and 
\begin{equation} \label{eqn-bar-Upsilon}
\bar\Upsilon_{\ell}(t)=\frac{\bar{S}_\ell(t)}{(\bar{B}_\ell(t))^\gamma} \sum_{\ell'=1}^L \beta_{\ell \ell'}\bar\mfF_{\ell'}(t), 
\end{equation}
where $\bar{B}_\ell = \bar{S}_\ell + \bar{I}_\ell + \bar{R}_\ell$ and $\bar{I}_\ell (t) = \bar{\mfI}_\ell(t,\infty)$. 
In addition, $(\bar{I}^N_\ell)_\ell \to (\bar{I}_\ell)_\ell$ in $\bD^L$ in probability where
\begin{align}
\bar{I}_\ell(t) &= \bar{I}_{\ell}(0)  -  \sum_{\ell'=1}^L   \int_0^{\infty}  \bigg( \int_0^t p_{\ell',\ell}(u)   F_0(du|\mfa)  \bigg)    \bar\mfI_{\ell'}(0,d\mfa)  +\int_{0}^t \bar\Upsilon_\ell(s) ds   \non \\
& \quad- \sum_{\ell'=1}^L \int_{0}^t \int_0^{t-s} p_{\ell',\ell}(u) F(du)  \bar\Upsilon_{\ell'}(s) ds +  \sum_{\ell'=1}^L \int_0^t \Big( \nu^I_{\ell', \ell}\bar{I}_{\ell'}(s) - \nu^I_{\ell, \ell'}\bar{I}_{\ell}(s)  \Big) ds\,.     \label{eqn-bar-I}    
\end{align} 
\end{theorem}

\bigskip

\section{The PDE model}  \label{sec-PDE} 

In this section we present the PDE model that is derived from the limiting integral equations. 
We assume that the distribution function $F$ is absolutely continuous, with a density $f$. For the extension of the results of this section to general $F$, see Remark \ref{GeneralF} below. 
 Recall that $\{X^\ell_i(t)\}_{i, \ell}$ and $\{X^{0,\ell}_j(t)\}_{j, \ell}$ are the Markov processes representing the migrations of newly and initially infected individuals, and have the same law, with transition probability functions $p_{\ell,\ell'}(\cdot)$ and transition rates $\nu^I_{\ell, \ell'}$. For notational convenience, we use the Markov process $\{X(t): t\ge 0\}$ to represent a typical migration process in the infected compartment.  Let $Q = (Q_{\ell,\ell'})$ denote the infinitesimal generator of $X(t)$, that is, 
\begin{align*}
Q_{\ell,\ell'}=\begin{cases}
                     \nu^I_{\ell,\ell'},&\text{ if $\ell'\not=\ell$},\\
                     -\sum_{\ell'\not=\ell}\nu^I_{\ell,\ell'},&\text{ if $\ell'=\ell$}.
                     \end{cases}
                     \end{align*}
Then the transition probability function  satisfies $p_{\ell,\ell'}(t)=\left(e^{Qt}\right)_{\ell,\ell'}$.

In case $\bar{\mfI}_\ell(t,\mfa)$ is absolutely continuous $\mfa$, we write  $ \bar\mfi_\ell(t,\mfa)= \frac{\partial}{\partial \mfa}\bar{\mfI}_\ell(t,\mfa)$ and consider $ \bar\mfi(t,\mfa) = (\bar\mfi_\ell(t,\mfa))_{\ell\in \sL}$  as a row vector. 
Let $\mu(\mfa)=f(\mfa)/F^c(\mfa)$ be the hazard rate function of the law of $\eta$. 

\begin{theorem} \label{thm-PDE}
Suppose that $\bar\mfI_\ell(0, \cdot)$ is absolutely continuous with density $\bar\mfi_\ell(0,\cdot)$ for each $\ell \in \sL$, and that $F$ is absolutely continuous with density $f$. 
 Then  $\bar\mfI_\ell(t, \mfa)$ is absolutely continuous in $\mfa$ and  $\bar\mfi(t,\mfa)$ is the unique solution to the following PDE: 
 \begin{align} \label{eqn-PDE} 
\frac{\partial}{\partial t}\bar\mfi(t,\mfa)+\frac{\partial}{\partial \mfa}\bar\mfi(t,\mfa)
&= -  \mu(\mfa) \, \bar\mfi(t,\mfa)   +\bar\mfi(t,\mfa)Q\,,
\end{align}
 for  $(t,\mfa)$ in $(0,\infty)^2$, 
where the initial condition is given by  $\bar{\mfi}(0,\mfa)$, and the boundary condition reads 
\begin{align} \label{eqn-bar-mfi-BC}
\bar\mfi_{\ell}(t,0) &= \frac{\bar{S}_\ell(t)}{ (\bar{B}_\ell(t))^\gamma} \sum_{\ell'=1}^L \beta_{\ell,\ell'} 
\int_0^{\infty}\bar\lambda(\mfa)\, \bar\mfi_{\ell'}(t,\mfa) \frac{F^c((\mfa-t)^+)}{F^c(\mfa)} d\mfa \,, 
\end{align} 
with $\bar{B}_\ell(t) = \bar{S}_\ell(t)+ \bar{I}_\ell(t) + \bar{R}_\ell(t)$. 
Moreover, the  PDE has a unique solution: 
 for $t\le\mfa$, 
\begin{equation}\label{ident-IC}
\bar\mfi(t,\mfa)= \frac{F^c(\mfa)}{F^c(\mfa-t)} \, \bar\mfi(0,\mfa-t) e^{Qt}\,,
\end{equation}
and for $t>\mfa$, 
\begin{equation}\label{ident-BC}
\bar\mfi(t,\mfa)=F^c(\mfa)\,  \bar\mfi(t-\mfa,0) e^{Q\mfa}\,,
\end{equation}
where  the boundary condition   $\bar{\mfi}(t,0)$ is the first component of the unique solution to the following set of integral equations: 
\begin{align}
\bar\mfi_{\ell}(t,0) &= \frac{\bar{S}_\ell(t)}{ (\bar{B}_\ell(t))^\gamma} \sum_{\ell'=1}^L \beta_{\ell,\ell'} 
\bigg( \sum_{\ell''=1}^L p_{\ell'',\ell'}(t) \int_0^{\infty} \bar{\lambda}(\mfa+t) \, \bar\mfi_{\ell''}(0,\mfa)d\mfa 
 \label{eqn-bar-mfi-BC-1}\\
 & \qquad \qquad \qquad \qquad \quad +  \sum_{\ell''=1}^L   \int_0^t \bar{\lambda}(t-s) p_{\ell'',\ell'}(t-s)  \, \bar\mfi_{\ell''}(s,0) ds  \bigg)\,, 
\non\\
\bar{S}_\ell(t) &= \bar{S}_\ell(0) - \int_0^t  \bar\mfi_{\ell}(s,0) ds + \sum_{\ell'=1}^L \int_0^t \Big( \nu^S_{\ell', \ell}\bar{S}_{\ell'}(s) - \nu^S_{\ell, \ell'}\bar{S}_{\ell}(s)  \Big) ds\,, \label{eqn-bar-S-BC}\\ 
\bar{I}_\ell(t) &= \bar{I}_{\ell}(0)  -  \sum_{\ell'=1}^L \int_0^{\infty}    \bigg( \int_0^t p_{\ell',\ell}(u)   \frac{f(u+\mfa)}{F^c(\mfa)}du \bigg)    \, \bar\mfi_{\ell'}(0,\mfa)d\mfa  +\int_{0}^t \bar\mfi_{\ell}(s,0) ds   \non \\
& \quad- \sum_{\ell'=1}^L \int_{0}^t \int_0^{t-s} p_{\ell',\ell}(u) f(u)du\,\, \bar\mfi_{\ell'}(s,0)  ds +  \sum_{\ell'=1}^L \int_0^t \Big( \nu^I_{\ell', \ell}\bar{I}_{\ell'}(s) - \nu^I_{\ell, \ell'}\bar{I}_{\ell}(s)  \Big) ds\,,    \label{eqn-bar-I-BC}    \\ 
\bar{R}_\ell(t) &= \bar{R}_\ell(0) +  \sum_{\ell'=1}^L   \int_0^{\infty} \bigg( \int_0^t p_{\ell',\ell}(u)   \frac{f(u+\mfa)}{F^c(\mfa)}du \bigg)  \, \bar\mfi_{\ell'}(0,\mfa)d\mfa \non\\
& \quad + \sum_{\ell'=1}^L \int_{0}^t \int_0^{t-s} p_{\ell',\ell}(u) f(u)du \,\, \bar\mfi_{\ell'}(s,0) ds  + \sum_{\ell'=1}^L \int_0^t \Big( \nu^R_{\ell', \ell}\bar{R}_{\ell'}(s) - \nu^R_{\ell, \ell'}\bar{R}_{\ell}(s)  \Big) ds\,.  \label{eqn-bar-R-BC}
\end{align} 

\end{theorem} 

\medskip

\begin{remark}
Consider the particular case where $Q=0$, i.e., $\nu^I_{\ell,\ell'}=0$ for all $\ell, \ell'$,  where individuals (at least the infected individuals) do not move.  In that case, 
the PDE in \eqref{eqn-PDE} simplifies to
\begin{equation} \label{eqn-PDE-1p}
\frac{\partial}{\partial t}\bar\mfi_\ell(t,\mfa)+\frac{\partial}{\partial \mfa}\bar\mfi_\ell(t,\mfa)=-\mu(\mfa)\bar\mfi_\ell(t,\mfa),
\end{equation}
for all $1\le \ell\le L$, where $\mu(\mfa)=f(\mfa)/F^c(\mfa)$ is the hazard rate function of the law of $\eta$. 
The formulas \eqref{ident-IC} and 
\eqref{ident-BC} reduce to  
\[ \bar\mfi_\ell(t,\mfa)=\frac{F^c(\mfa)}{F^c(\mfa-t)}\bar\mfi_\ell(0,\mfa-t)\,\text{ if }t\le \mfa,\, \text{ and }\quad \bar\mfi_\ell(t,\mfa)=F^c(\mfa)\bar\mfi_\ell(t-\mfa,0)
,\text{ if }\mfa< t\,,\]
for all $1\le \ell\le L$,
which are exactly the formulas for the homogeneous model in our previous work \cite{PP-2021}.  
\end{remark}

\begin{remark}
In the special case where $\lambda(t) = \tilde{\lambda}(t) {\bf1}_{t <\eta}$, with a deterministic function $\tilde{\lambda}(t)$, we have $\bar\lambda(t) =\E[\lambda^\ell_i(t)]= \tilde{\lambda}(t) F^c(t)$ for each $\ell,i$, and $\E[\lambda^{0,\ell}_j(t)|\tilde\tau^{\ell,N}_{j,0}=\mfa] = \tilde{\lambda}(t+\mfa)\frac{ F^c(t+\mfa)}{F^c(\mfa)}$ for each $\ell,j$. In this case, the boundary condition in \eqref{eqn-bar-mfi-BC} becomes
\begin{align}
\bar\mfi_{\ell}(t,0) = \frac{\bar{S}_\ell(t)}{ (\bar{B}_\ell(t))^\gamma} \sum_{\ell'=1}^L \beta_{\ell,\ell'} 
\int_0^{\infty}\tilde\lambda(\mfa)\, \bar\mfi_{\ell'}(t,\mfa)  d\mfa \,. 
\end{align} 
This is consistent with the formula in the homogeneous model (see Remark 3.3 in \cite{PP-2021}). It is also  how the boundary conditions for some PDE epidemic models  in the literature are usually formulated (see, e.g., \cite{inaba2004mathematical,magal2013two,foutel2020individual}). 
\end{remark}

\begin{proof}
We first derive the PDE model. Recall the expression of $\bar{\mfI}_\ell(t,\mfa) $ in \eqref{eqn-bar-mfI}.  Since both $\bar{\mfI}_\ell(0,\cdot)$ (for each $\ell$) and $F$ are absolutely continuous, then $\bar\mfI_\ell(t,\mfa)$ is differentiable in $t$ and $\mfa$, and we have
\begin{align} \label{eqn-PDE-mfI}
\frac{\partial}{\partial t}\bar\mfI_\ell(t,\mfa)+\frac{\partial}{\partial \mfa}\bar\mfI_\ell(t,\mfa)
&=-\sum_{\ell'=1}^L\int_0^{(\mfa-t)^+}p_{\ell',\ell}(t)\frac{f(t+y)}{F^c(y)}\bar\mfI_{\ell'}(0,dy)
+\bar\Upsilon_\ell(t) \non 
\\&\quad-\sum_{\ell'=1}^L\int_{(t-\mfa)^+}^t p_{\ell',\ell}(t-s)f(t-s)\bar\Upsilon_{\ell'}(s)ds
+\sum_{\ell'=1}^L\left[\nu^I_{\ell',\ell}\bar\mfI_{\ell'}(t,\mfa)-\nu^I_{\ell,\ell'}\bar\mfI_\ell(t,\mfa)\right]\,.
\end{align}
We also note that for $0<\mfa<t$ and $\mfa$ small,  $\frac{\partial}{\partial \mfa}\bar\mfI_\ell(t,\mfa)=\bar\Upsilon_\ell(t-\mfa)+O(\mfa)$, and consequently, letting $\mfa\to0$, we deduce that
\begin{align} \label{eqn-Upsilon-mfi}
 \bar\Upsilon_\ell(t)  =  \bar\mfi_{\ell}(t,0)\,. 
\end{align} 
We differentiate \eqref{eqn-PDE-mfI} with respect to $\mfa$, at least in the distributional sense, and deduce the following identify from the fact that $\frac{\partial}{\partial \mfa}\frac{\partial}{\partial t}\bar\mfI_\ell(t,\mfa)=\frac{\partial}{\partial t}\frac{\partial}{\partial \mfa}\bar\mfI_\ell(t,\mfa)$: 
\begin{align}\label{PDE-first}
\frac{\partial}{\partial t}\bar\mfi_\ell(t,\mfa)+\frac{\partial}{\partial \mfa}\bar\mfi_\ell(t,\mfa)
&=
-{\bf1}_{t\le \mfa}\sum_{\ell'=1}^Lp_{\ell',\ell}(t)\frac{f(\mfa)}{F^c(\mfa-t)}\bar\mfi_{\ell'}(0,\mfa-t) \non \\&\quad
-{\bf1}_{\mfa< t}\sum_{\ell'=1}^Lp_{\ell',\ell}(\mfa)f(\mfa)\bar\mfi_{\ell'}(t-\mfa,0)+\sum_{\ell'=1}^L\left[\nu^I_{\ell',\ell}\bar\mfi_{\ell'}(t,\mfa)-\nu^I_{\ell,\ell'}\bar\mfi_\ell(t,\mfa)\right]\,.
\end{align}

We next obtain a relation between $\mfi(t,\mfa)$ and $\mfi(0,\mfa-t)$ or $\mfi(t-\mfa,0)$, depending upon whether $t\le \mfa$ or $\mfa< t$, which will lead to the expressions of $\bar\mfi(t,\mfa)$ in \eqref{ident-IC} and \eqref{ident-BC}.   We start with the first case
$t\le \mfa$. For  $0 \le s \le t$, by \eqref{PDE-first}, 
\begin{align} \label{eqn-PDE-first-1}
\frac{d}{ds}\bar\mfi_\ell(s,\mfa-t+s)&=\Big(\frac{\partial}{\partial t}+\frac{\partial}{\partial \mfa}\Big)\bar\mfi_\ell(s,\mfa-t+s)\nonumber \\
&=-\sum_{\ell'=1}^L p_{\ell',\ell}(s)\frac{f(\mfa-t+s)}{F^c(\mfa-t)}\bar\mfi_{\ell'}(0,\mfa-t)+(\bar\mfi(s,\mfa-t+s)Q)_\ell\,.
\end{align}
The value at time $s=t$ of the solution of this linear system of ODEs is given by  \eqref{ident-IC}, that is, 
\begin{equation*}
\bar\mfi(t,\mfa)=\frac{F^c(\mfa)}{F^c(\mfa-t)} \bar\mfi(0,\mfa-t)e^{Qt}\,.
\end{equation*}
To see this, by letting $p(t) = (p_{\ell,\ell'}(t))=e^{Qt}$, $y_s = \bar\mfi_\ell(s,\mfa-t+s)$, and $\gamma_s = \frac{f(\mfa-t+s)}{F^c(\mfa-t)}$, the equation \eqref{eqn-PDE-first-1} can be written as 
\begin{align*}
\dot{y}_s  & = -\gamma_s y_0  p(s) +y_s Q \\
& = -\gamma_s y_0 e^{sQ} +y_s Q\,.
\end{align*}
By the Duhamel formula, we obtain the following solution to this linear ODE:
\begin{align*}
y_t  &= y_0 e^{tQ} - y_0 \int_0^t \gamma_s e^{sQ} e^{(t-s)Q}ds \\
&=  \Big( 1- \int_0^t \gamma_sds \Big) y_0 e^{tQ} \\
&= \frac{F^c(\mfa)}{F^c(\mfa-t)} \, y_0 e^{tQ}\,.
\end{align*}

We then consider the case $\mfa < t$. 
For $0 \le s \le \mfa$,  by \eqref{PDE-first}, 
\begin{align*}
\frac{d}{ds}\bar\mfi_\ell(t-\mfa+s,s)&=\Big(\frac{\partial}{\partial t}+\frac{\partial}{\partial \mfa}\Big)\bar\mfi_\ell(t-\mfa+s,s)\\
&=-\sum_{\ell'=1}^L p_{\ell',\ell}(s)f(s)\bar\mfi_{\ell'}(t-\mfa,0)+(\bar\mfi_\ell(t-\mfa+s,s)Q)_\ell\,.
\end{align*}
The value at time $s=\mfa$ of the solution of this linear system of ODEs is given by 
\eqref{ident-BC}, that is,
\begin{equation*}
\bar\mfi(t,\mfa)= F^c(\mfa) \bar\mfi(t-\mfa,0)  e^{Q\mfa}\,.
\end{equation*}

Thus, by \eqref{PDE-first} and these two identities, we obtain the PDE in \eqref{eqn-PDE}.

We then derive the boundary condition.  By \eqref{eqn-Upsilon-mfi} and \eqref{eqn-bar-Upsilon}, using \eqref{eqn-bar-mfF}, we  obtain the boundary condition expression \eqref{eqn-bar-mfi-BC-1} for $\bar\mfi_{\ell}(t,0)$: 
\begin{align*}
\bar\mfi_{\ell}(t,0) &= \frac{\bar{S}_\ell(t)}{ (\bar{B}_\ell(t))^\gamma} \sum_{\ell'=1}^L \beta_{\ell,\ell'} 
\bigg( \sum_{\ell''=1}^L p_{\ell'',\ell'}(t) \int_0^{\infty} \bar{\lambda}(\mfa+t) \, \bar\mfi_{\ell''}(0,\mfa)d\mfa 
 \label{eqn-bar-mfi-BC-1}\\
 & \qquad \qquad \qquad \qquad \quad +  \sum_{\ell''=1}^L   \int_0^t \bar{\lambda}(t-s) p_{\ell'',\ell'}(t-s)  \, \bar\mfi_{\ell''}(s,0) ds  \bigg)\,. 
 \end{align*} 

We  rewrite the first integral on the right (in vector form) as follows
\begin{align*}
\int_0^{\infty} \bar{\lambda}(\mfa+t) \, \bar\mfi(0,\mfa)e^{Qt}d\mfa&=\int_t^{\infty}\bar\lambda(u)\bar\mfi(0,u-t)e^{Qt}du\\
&=\int_t^{\infty}\bar\lambda(u)\bar\mfi(t,u) \frac{F^c(u-t)}{F^c(u)}du\,,
\end{align*}
where in the second equality we have used \eqref{ident-IC}.
We rewrite the second integral as follows
\begin{align*}
\int_0^t\bar\lambda(t-s)\bar\mfi(s,0)e^{Q(t-s)}ds&=\int_0^t\bar\lambda(u)\bar\mfi(t-u,0)e^{Qu}du\\
&=\int_0^t\bar\lambda(u)\bar\mfi(t,u) \frac{1}{F^c(u)} du\\
&=\int_0^t\bar\lambda(u)\bar\mfi(t,u) \frac{F^c(u-t)}{F^c(u)} du\,,
\end{align*}
where in the second equality we have used \eqref{ident-BC}, and the fact that $F^c=1$ on $\mathbb{R}_-$.
From these we obtain the boundary condition expression of $\bar\mfi_{\ell}(t,0)$
 in \eqref{eqn-bar-mfi-BC}.  

In addition, the expressions of $\bar{S}_\ell(t)$ in \eqref{eqn-bar-S-BC}, $\bar{I}_\ell(t)$ in \eqref{eqn-bar-I-BC} and 
$\bar{R}_\ell(t)$ in \eqref{eqn-bar-R-BC}, are obtained from the equations in \eqref{eqn-barS}, \eqref{eqn-bar-I} and \eqref{eqn-barR} by replacing $ \bar\Upsilon_\ell(t)  =  \bar\mfi_{\ell}(t,0) $ and using the density $\bar\mfi_{\ell}(0,\mfa)$. 

It is then clear that existence and uniqueness of the solution to \eqref{eqn-PDE}  follows from the existence and uniqueness of the solution  to the system of equations satisfied by the boundary condition, $\bar{S}$,  
$\bar{I}$ and $\bar{R}$ (Lemma \ref{uniqBC} below), as well as the explicit expressions of the PDEs \eqref{eqn-PDE} in both cases $\mfa\ge t$ and $\mfa< t$ in terms of the initial conditions and boundary conditions in \eqref{ident-IC} and \eqref{ident-BC}.
\end{proof}

We next show that there exists a unique solution to the boundary conditions determined by the set of equations in  \eqref{eqn-bar-mfi-BC-1}--\eqref{eqn-bar-R-BC}.

\begin{lemma}\label{uniqBC}
The system of integral equations \eqref{eqn-bar-mfi-BC-1}, \eqref{eqn-bar-S-BC},  \eqref{eqn-bar-I-BC}
and \eqref{eqn-bar-R-BC} has a unique solution in $\bC(\R_+;\R_+^{4L})$.
\end{lemma}

\begin{proof}
We consider the cases of $\gamma=0$ and $\gamma \in (0,1]$ separately. When $\gamma=0$, the set of equations reduces to the systems of linear Volterra equations of $\bar\mfi_{\ell}(t,0)$ and $\bar{S}_\ell(t)$, that is,
\begin{align*} 
\bar\mfi_{\ell}(t,0) &= \bar{S}_\ell(t) \sum_{\ell'=1}^L \beta_{\ell,\ell'} 
\bigg( \sum_{\ell''=1}^L p_{\ell'',\ell'}(t) \int_0^{\infty} \bar{\lambda}(\mfa+t) \, \bar\mfi_{\ell''}(0,\mfa)d\mfa 
 \non\\
 & \qquad \qquad \qquad \qquad +  \sum_{\ell''=1}^L   \int_0^t \bar{\lambda}(t-s) p_{\ell'',\ell'}(t-s)  \, \bar\mfi_{\ell''}(s,0) ds  \bigg)\,,
\end{align*}
and
\begin{align*}
\bar{S}_\ell(t) &= \bar{S}_\ell(0) - \int_0^t  \bar\mfi_{\ell}(s,0) ds + \sum_{\ell'=1}^L \int_0^t \Big( \nu^S_{\ell', \ell}\bar{S}_{\ell'}(t) - \nu^S_{\ell, \ell'}\bar{S}_{\ell}(t)  \Big) ds\,. 
\end{align*}
Thus, the existence and uniqueness of a solution follow from the well known theory of linear Volterra integral equations (see, e.g., \cite{brunner2017volterra}). 

We next consider the case $\gamma \in (0,1]$. 
 Define 
\[ \bar{V}_\ell(t)=\bar{\mfi}(t,0)\frac{(\bar{B}_\ell(t))^\gamma}{\bar{S}_\ell(t)}, \quad 1\le \ell\le L, \ t\ge0.\] 
Let moreover
\begin{align*}
 f_\ell(t)&=\sum_{\ell', \ell''=1}^L\beta_{\ell,\ell'}p_{\ell'',\ell'}(t)\int_0^{\infty}\bar{\lambda}(\mfa+t)\, \bar{\mfi}_{\ell''}(0,\mfa)d\mfa\, ,\\
 g_\ell(t)&=\bar{I}_\ell(0)-\sum_{\ell'=1}^L \int_0^{\infty} \bigg( \int_0^t p_{\ell',\ell}(u)   \frac{f(u+\mfa)}{F^c(\mfa)}du \bigg)   \, \bar{\mfi}_{\ell'}(0,\mfa)d\mfa,\\
 h_\ell(t)&=\bar{R}_\ell(0)+\sum_{\ell'=1}^L\int_0^{\infty} \bigg( \int_0^t p_{\ell',\ell}(u)   \frac{f(u+\mfa)}{F^c(\mfa)}du \bigg)   \, \bar{\mfi}_{\ell'}(0,\mfa)d\mfa,\\
H_{\ell,\ell'}(t)&=\sum_{\ell''=1}^L\beta_{\ell,\ell''}\bar{\lambda}(t)p_{\ell',\ell''}(t),\\
G_{\ell,\ell'}(t)&=\sum_{\ell'=1}^L\int_0^tp_{\ell',\ell}(u)f(u)du\,.
 \end{align*}
With these notations, the system of equations  \eqref{eqn-bar-mfi-BC-1}, \eqref{eqn-bar-S-BC},  \eqref{eqn-bar-I-BC}
and \eqref{eqn-bar-R-BC} can be rewritten as
\begin{align*}
\bar{V}_\ell(t)&= f_\ell(t)+\sum_{\ell'=1}^L\int_0^tH_{\ell,\ell'}(t-s)\frac{\bar{S}_{\ell'}(s)}{(\bar{B}_{\ell'}(s))^\gamma}
\bar{V}_{\ell'}(s)ds,\\
\bar{S}_\ell(t)&=\bar{S}_\ell(0)-\int_0^t \frac{\bar{S}_{\ell}(s)}{(\bar{B}_{\ell}(s))^\gamma}\bar{V}_\ell(s)ds
+ \sum_{\ell'=1}^L \int_0^t \Big( \nu^S_{\ell', \ell}\bar{S}_{\ell'}(s) - \nu^S_{\ell, \ell'}\bar{S}_{\ell}(s)  \Big) ds,\\
\bar{I}_\ell(t)&=g_\ell(t) + \int_0^t \frac{\bar{S}_{\ell}(s)}{(\bar{B}_{\ell}(s))^\gamma}\bar{V}_\ell(s)ds -\sum_{\ell'=1}^L\int_0^tG_{\ell,\ell'}(t-s)\frac{\bar{S}_{\ell'}(s)}{(\bar{B}_{\ell'}(s))^\gamma}
\bar{V}_{\ell'}(s)ds \\
& \qquad \qquad  +\sum_{\ell'=1}^L \int_0^t \Big( \nu^I_{\ell', \ell}\bar{I}_{\ell'}(s) - \nu^I_{\ell, \ell'}\bar{I}_{\ell}(s)  \Big) ds,\\
\bar{R}_\ell(t)&=h_\ell(t)+\sum_{\ell'=1}^L\int_0^tG_{\ell,\ell'}(t-s)\frac{\bar{S}_{\ell'}(s)}{(\bar{B}_{\ell'}(s))^\gamma}
\bar{V}_{\ell'}(s)ds+\sum_{\ell'=1}^L \int_0^t \Big( \nu^R_{\ell', \ell}\bar{R}_{\ell'}(s) - \nu^R_{\ell, \ell'}\bar{R}_{\ell}(s)  \Big) ds\, .
\end{align*}

In order to deduce existence and uniqueness of a unique solution of this system of $4L$ equations from standard results
on integral equations (see, e.g., \cite{brunner2017volterra}), it suffices to show that 
$\frac{\bar{S}_{\ell}(s)}{(\bar{S}_{\ell}(s)+\bar{I}_{\ell}(s)+\bar{R}_{\ell}(s))^\gamma}$ is a bounded and uniformly Lipschitz function of its three arguments. 

By \eqref{eqn-bar-S-BC}, \eqref{eqn-bar-I-BC} and \eqref{eqn-bar-R-BC},
we have
\begin{align*}
\bar{B}_\ell(t) & = \bar{B}_\ell(0)  + \sum_{\ell'=1}^L \int_0^t \Big( \nu^S_{\ell', \ell}\bar{S}_{\ell'}(t) - \nu^S_{\ell, \ell'}\bar{S}_{\ell}(t)  \Big) ds +  \sum_{\ell'=1}^L \int_0^t \Big( \nu^I_{\ell', \ell}\bar{I}_{\ell'}(t) - \nu^I_{\ell, \ell'}\bar{I}_{\ell}(t)  \Big) ds \\
& \qquad +  \sum_{\ell'=1}^L \int_0^t \Big( \nu^R_{\ell', \ell}\bar{R}_{\ell'}(t) - \nu^R_{\ell, \ell'}\bar{R}_{\ell}(t)  \Big) ds \\
& \ge  \bar{B}_\ell(0)  + \sum_{\ell'=1}^L \int_0^t \Big( \underline{\nu}_{\ell', \ell}\bar{B}_{\ell'}(t) - \bar{\nu}_{\ell, \ell'}\bar{B}_{\ell}(t)  \Big) ds  \\
& =  \bar{B}_\ell(0)  -  \int_0^t \Big(\sum_{\ell'=1}^L  \bar{\nu}_{\ell, \ell'} \Big) \bar{B}_{\ell}(t)  ds + \sum_{\ell'=1}^L \int_0^t  \underline{\nu}_{\ell', \ell}\bar{B}_{\ell'}(t) ds 
\end{align*} 
where $\underline{\nu}_{\ell', \ell} =\min \{\nu^S_{\ell', \ell}, \nu^I_{\ell', \ell},\nu^R_{\ell', \ell} \}$ and 
$\bar{\nu}_{\ell', \ell} =\max \{\nu^S_{\ell', \ell}, \nu^I_{\ell', \ell},\nu^R_{\ell', \ell} \}$. 
From this we can deduce that there exists a constant $c_T>0$ such that 
for each $0 \le t \le T$, $1\le\ell\le L$, $\bar{B}_\ell(t) >c_T$.  Then the boundedness and uniform Lipschitz properties follow easily. This completes the proof of the lemma. 
\end{proof}

\begin{remark}\label{GeneralF}
We can follow a similar argument as in \cite{PP-2021} for the homogeneous model to derive the PDE model when the distribution function $F$ is not necessarily absolutely continuous. We replace $f(x)dx$ by $F(dx)$. 
Then the PDE in \eqref{eqn-PDE} becomes
\begin{align} \label{eqn-PDE-general} 
\frac{\partial}{\partial t}\bar\mfi(t,\mfa)+\frac{\partial}{\partial \mfa}\bar\mfi(t,\mfa)
&=- \frac{F(d\mfa)}{F^c(\mfa^-)} \, \bar\mfi(t,\mfa) +\bar\mfi(t,\mfa)Q\,. 
\end{align}
The boundary condition can be modified accordingly. We omit the details here. 
\end{remark}

\bigskip

\section{Proof of the FLLN} \label{sec-proof-FLLN} 

\subsection{Convergence of $\bar\mfF^N_\ell$}

Recall the expression of $A^{N}_\ell$ in \eqref{eqn-An} and the instantaneous infectivity rate function $\Upsilon^{N}_{\ell}$ in \eqref{eqn-Upsilon}.  The process $A^{N}_\ell$ has the semimartingale decomposition
\begin{align} \label{eqn-An-decomp}
A^{N}_\ell(t) = M_{A, \ell}^{N}(t) + \int_0^t \Upsilon^{N}_\ell(s)ds,
\end{align}
where 
\begin{equation} \label{eqn-MAn}
M_{A, \ell}^{N}(t) = \int_0^t \int_0^\infty {\bf1}_{u \le \Upsilon^{N}_\ell(s^-)} \overline{Q}_{\ell}(ds, du)
\end{equation}
and $\overline{Q}_{\ell}(ds, du) := Q_{\ell}(ds,du) - ds du$ is the compensated PRM associated with $Q_{\ell}(ds,du) $. 
It can be shown that $\{M_{A, \ell}^{N}(t):  t\ge 0\} $ is a square-integrable martingale with respect to the filtration $\{\sF^N_t: t\ge 0\}$ where
\begin{align*}
\sF^N_t & := \sigma \bigg\{S_\ell^{N}(0), I_\ell^{N}(0), R_\ell^{N}(0), \mfI^N_\ell(0, \cdot), \{\lambda^{0,\ell}_{j}(\cdot)\}_{j\ge 1},  \{\lambda^{\ell}_{i}(\cdot)\}_{i\ge 1},   \ell \in \sL \bigg\} \\
& \qquad \quad  \vee \sigma \bigg\{  \int_0^{t'} \int_0^\infty \bone_{u \le \Upsilon^{N}_\ell(s^-)} Q_{\ell}(ds, du): t' \le t,  \ell \in \sL \bigg\}. 
\end{align*}
See, e.g., \cite[Chapter VI]{ccinlar2011probability}. The martingale $M_{A, \ell}^{N}(t)$ has a finite quadratic variation 
$$
\langle M_{A, \ell}^{N} \rangle(t)  = \int_0^t  \Upsilon^{N}_\ell(s)ds, \quad t\ge 0,
$$
which satisfies
\begin{align} \label{eqn-Upsilon-N-bound}
0 \le   \int_s^t \bar \Upsilon^{N}_\ell(u)du \le  \lambda^*  \beta^* (t-s), \quad \text{w.p.1 \quad for } \quad 0 \le s \le t.  
\end{align}
Since $\langle\bar{M}^N_{A,\ell}\rangle(t)= N^{-1}\int_0^t\bar{\Upsilon}^N_\ell(s)ds \le N^{-1}  \lambda^*  \beta^* t$, from Doob's inequality we deduce that locally uniformly in $t$, 
\begin{align} \label{eqn-barMn-conv}
\bar{M}_{A, \ell}^{N}(t) \to 0  
\end{align} 
in probability as $N\to\infty$, 
and as a consequence, the following lemma holds (whose proof is very similar to that of Lemma 4.1 in \cite{FPP2020b}). 
\begin{lemma} \label{lem-An-conv}
For each  $\ell \in \sL$,  the sequence $\{\bar{A}^{N}_\ell: N \in \NN\}$ is tight in $\bD$, and the limit of each convergent subsequence of $\{\bar{A}^{N}_\ell\}$, denoted by $\bar{A}_{\ell}$, satisfies
\begin{equation*}
\bar{A}_{\ell} = \lim_{N\to\infty} \bar{A}^{N}_\ell = \lim_{N\to\infty}  \int_0^{\cdot} \bar\Upsilon^{N}_\ell(s)ds,
\end{equation*}
and
\begin{equation} \label{eqn-barA-inc} 
0 \le \bar{A}_{\ell}(t) - \bar{A}_{\ell}(s) \le  \lambda^* \beta^* (t-s), \quad \text{w.p.1} \quad \text{for} \quad 0 \le s \le t. 
\end{equation}
 \end{lemma}

It clearly follows from the last inequality that for each $\ell\in\mathcal{L}$, the measure whose distribution function is the increasing function $\bar{A}_{\ell}(t)$ is absolutely continuous with respect to  Lebesgue's measure. In fact, since the sequence $\bar\Upsilon^{N}_\ell$ is bounded in $L^2(0,T)$ for any $T>0$, the above converging subsequence is such that $\bar\Upsilon^{N}_\ell$ converges in law in $L^2(0,T)$ equipped with its weak topology. But we do not know yet that its limit is the function $\bar\Upsilon_\ell$ given by \eqref{eqn-bar-Upsilon}. 

\medskip

Recall $\mfF^N_\ell(t)$ in \eqref{eqn-mfF}. Let 
\begin{align}
\bar{\mfF}^{N,0}_\ell (t) &:= N^{-1}  \sum_{\ell'=1}^L \sum_{j=1}^{I^N_{\ell'}(0)} \lambda^{0,\ell'}_j(\tilde{\tau}_{j,0}^{\ell',N}+t) {\bf1}_{X^{0, \ell'}_j(t) = \ell}  \,\,, \label{eqn-bar-mfFn-0} \\
\bar{\mfF}^{N,1}_\ell(t) &:= N^{-1} \sum_{\ell'=1}^L\sum_{i=1}^{A^N_{\ell'}(t)} \lambda_i^{\ell'}(t-\tau^{\ell',N}_{i}) {\bf1}_{X^{ \ell'}_i(t-\tau^{\ell',N}_{i}) = \ell}  \,\,. 
\label{eqn-bar-mfFn-1}
 \end{align}

  \begin{lemma} \label{lem-bar-mfFn-conv}
 Under Assumptions \ref{AS-initial} and \ref{AS-lambda}, along any convergent subsequence of  $\{\bar{A}^{N}_\ell\}$ with the limit $\{\bar{A}_{\ell}\}$ for each $\ell \in \sL$, 
 \begin{equation*}
\big( \bar{\mfF}^{N}_\ell \big)_{  \ell \in \sL}  \Rightarrow \big(\bar{\mfF}_\ell \big)_{  \ell \in \sL}  \qinq \bD^L \end{equation*}
 as $N\to\infty$, 
where $\bar{\mfF}_\ell = \bar{\mfF}^{0}_\ell  + \bar{\mfF}^{1}_\ell$ with 
\begin{equation} \label{eqn-bar-mfI-0-def}
\bar{\mfF}^{0}_\ell (t) := \sum_{\ell'=1}^L p_{\ell',\ell}(t) \int_0^{\infty} \bar{\lambda}(\mfa+t) \bar\mfI_{\ell'}(0,d\mfa), \quad t\ge 0,
\end{equation}
and
\begin{equation} \label{eqn-bar-mfI-1-def}
\bar{\mfF}^{1}_\ell (t) := \sum_{\ell'=1}^L   \int_0^t \bar{\lambda}(t-s) p_{\ell',\ell}(t-s)  d \bar{A}_\ell(s)\,, \quad t\ge 0. 
\end{equation}
 \end{lemma}
 
 Note that the limit $\bar{\mfF}_\ell$ is not yet the same as that given in \eqref{eqn-bar-mfF} since $\bar{A}_{\ell'}(ds)$ in  \eqref{eqn-bar-mfI-1-def} remains to be identified as $\bar\Upsilon_{\ell'}(s)ds$. So we are abusing the notation to use  $\bar{\mfF}_\ell$ in this lemma.  The proof of this lemma follows from a slight modification of the proof approach in \cite{FPP-2022} to take into account the difference in the initial condition, which is omitted for brevity.  The pointwise convergence is part of the proof of the crucial Lemma 4.3 in \cite{FPP-2022}, and the convergence in $\bD^L$ is then obtained in the first part of subsection 4.5. We remark that the approach  in Section 4 of \cite{FPP-2022} uses an argument adopted from the ``propagation of chaos" for interacting particle systems \cite{sznitman1991topics} which requires only the conditions $\lambda(\cdot) \in \bD$ a.s. and $\sup_{t \ge0}  \lambda(t)\le \lambda^\ast$, instead of the regularity conditions as stated in Assumption 2.1 in \cite{FPP2020b}.

\subsection{Convergence of $(\bar{S}^N_\ell, \bar{\mfI}^N_\ell, \bar{R}^N_\ell)_{\ell\in\sL}$} 

We first have the following representations of the LLN-scaled processes, from \eqref{eqn-mfF}--\eqref{eqn-mfI-N-new}:
\begin{align} \label{eqn-bar-Sn}
\bar{S}^{N}_{\ell}(t) &= \bar{S}^{N}_{\ell}(0) - \bar{A}^N_\ell(t) 
+ \sum_{\ell'=1}^L \big( \bar{M}^{N}_{S, \ell', \ell} (t) - \bar{M}^{N}_{S,\ell, \ell'} (t)\big)  \non\\
& \qquad + \sum_{\ell'=1}^L \bigg( \nu_{\ell',\ell}^{S} \int_0^t \bar{S}^{N}_{\ell'}(s)ds- \nu_{\ell,\ell'}^{S} \int_0^t \bar{S}^{N}_{\ell}(s)ds \bigg),
\end{align}
\begin{align} \label{eqn-bar-mfI-N} 
\bar\mfI^N_\ell(t, \mfa) 
&  = \bar\mfI^N_{\ell}(0,(\mfa-t)^+ ) - \bar\mfI^{N,0}_\ell(t, \mfa)    + \bar{A}^N_\ell(t) - \bar{A}^N_\ell((t-\mfa)^+) - \bar\mfI^{N,1}_\ell(t, \mfa)  \non\\
& \qquad  + \sum_{\ell'=1}^L \big( \bar{M}^{N}_{\mfI, \ell', \ell} (t,\mfa) - \bar{M}^{N}_{\mfI,\ell, \ell'} (t,\mfa)\big)    \\
& \qquad + \sum_{\ell'=1}^L \bigg( \nu_{\ell',\ell}^{I} \int_{(t-\mfa)^+}^t \bar\mfI^N_{\ell'}(s, \mfa-(t-s)) ds- \nu_{\ell,\ell'}^{I} \int_{(t-\mfa)^+}^t \bar\mfI^N_{\ell}(s, \mfa-(t-s)) ds \bigg), \non
\end{align}
\begin{align} \label{eqn-bar-Rn} 
\bar{R}^N_\ell(t)
& =  \bar{R}^N_\ell(0)+  \bar{R}^{N,0}_\ell(t) + \bar{R}^{N,1}_\ell(t)  + \sum_{\ell'=1}^L \big( \bar{M}^{N}_{R, \ell', \ell} (t) - \bar{M}^{N}_{R,\ell, \ell'} (t)\big)  \non \\
& \qquad + \sum_{\ell'=1}^L \bigg( \nu_{\ell',\ell}^{R} \int_0^t \bar{R}^{N}_{\ell'}(s)ds- \nu_{\ell,\ell'}^{R} \int_0^t \bar{R}^{N}_{\ell}(s)ds \bigg),
\end{align} 
and
\begin{align}
\bar\Upsilon^N_\ell(t)  =  \frac{\bar{S}_\ell^N(t)}{(\bar{B}_\ell^N(t))^\gamma} \sum_{\ell'=1}^L \beta_{\ell \ell'}\bar\mfF^N_{\ell'}(t)\,,
\end{align} 
where  $\bar{B}_\ell^N(t) = \bar{S}_\ell^N(t) + \bar{I}_\ell^N(t) + \bar{S}_\ell^N(t)$ with $\bar{I}^N_\ell(t) = \bar\mfI^N_\ell(t, \infty)$, 
$\bar{M}^{N}_{A,\ell} (t)$ is given in \eqref{eqn-MAn}, 
\begin{align*} 
\bar{M}^{N}_{Z, \ell, \ell'} (t) &= \frac{1}{N} \left( P^Z_{\ell,\ell'}\left( \nu^{Z}_{\ell, \ell'} \int_0^t Z^N_{\ell}(s) ds\right) -  \nu^{Z}_{\ell, \ell'} \int_0^t Z^N_{\ell}(s) ds   \right)\,, \quad Z=S, R, \\
\bar{M}^{N}_{\mfI, \ell, \ell'} (t,\mfa) &= \frac{1}{N} \left(\int_{(t-\mfa)^+}^t\int_0^\infty{\bf1}_{u\le \nu^{I}_{\ell, \ell'}\mfI^N_{\ell}(s, \mfa-(t-s))}Q^I_{\ell,\ell'}(ds,du)   -  \nu^{I}_{\ell, \ell'} \int_{(t-\mfa)^+}^t \mfI^N_{\ell}(s, \mfa-(t-s))  ds   \right)\,, 
\end{align*}
\begin{align*}
\bar\mfI^{N,0}_\ell(t, \mfa) & = \frac{1}{N}  \sum_{\ell'=1}^L  \sum_{j=1}^{\mfI^N_{\ell'}(0,(\mfa-t)^+ )} {\bf1}_{\eta^{0,\ell'}_j \le t} {\bf1}_{X_j^{0, \ell'}(\eta^{0,\ell'}_j ) = \ell}\,,\\
\bar\mfI^{N,1}_\ell(t, \mfa) & = \frac{1}{N} \sum_{\ell'=1}^L \sum_{i=A^N_{\ell'}((t-\mfa)^+)+1}^{A^N_{\ell'}(t)} {\bf1}_{\tau_{i}^{\ell',N}+\eta^{\ell'}_i \le t}  {\bf1}_{X_i^{\ell'}(\eta^{\ell'}_i) = \ell}\,. 
\end{align*}  
\begin{align*}
\bar{R}^{N,0}_\ell(t) &= \frac{1}{N} \sum_{\ell'=1}^L \sum_{j=1}^{I^N_{\ell'}(0)} {\bf1}_{\eta^{0,\ell'}_j \le t}{\bf1}_{X_j^{0, \ell'}(\eta^{0,\ell'}_j) = \ell} \, =\,\bar\mfI^{N,0}_\ell(t, \infty)  \,,\\
 \bar{R}^{N,1}_\ell(t) &= \frac{1}{N}   \sum_{\ell'=1}^L \sum_{i=1}^{A^N_{\ell'}(t)} {\bf1}_{\tau_{i}^{\ell',N}+\eta^{\ell'}_i \le t}  {\bf1}_{X_i^{\ell'}(\eta^{\ell'}_i) = \ell} \, =\,  \bar\mfI^{N,1}_\ell(t, \infty) \,. 
\end{align*}

\begin{lemma} \label{lem-barM-conv}
Under Assumptions \ref{AS-initial} and \ref{AS-lambda}, for each $\ell, \ell'\in \sL$, 
\begin{align}
\big(\bar{M}^{N}_{S, \ell, \ell'} (t), \bar{M}^{N}_{\mfI, \ell, \ell'} (t,\mfa), \bar{M}^{N}_{R, \ell, \ell'} (t) \big) \to 0 
\end{align}
in probability, uniformly in $t$ and $\mfa$, as $N\to\infty$. 
\end{lemma}
\begin{proof}
The process $\bar{M}^{N}_{S, \ell, \ell'} (t)$  is a square-integrable martingale with respect to the filtration:
\begin{align*}
\mathcal{F}^N_{S}(t) =\sF^N_t  \vee \sigma \bigg\{  P^S_{\ell,\ell'}\left( \nu^{S}_{\ell, \ell'} \int_0^{t'} S^N_{\ell}(s) ds\right):  t' \le t,\,   \ell,\ell' \in \sL \bigg\}, 
\end{align*}
with quadratic variation
\[
\langle \bar{M}^{N}_{S, \ell, \ell'} \rangle(t)= \frac{1}{N}\nu^{S}_{\ell, \ell'} \int_0^{t} \bar{S}^N_{\ell}(s) ds\le  \frac{1}{N}\nu^{S}_{\ell, \ell'} \sum_{\ell\in\sL} \bar{S}^N_{\ell}(0) t \,, 
\]
which converges to zero in probability as $N\to \infty$. This implies that $\bar{M}^{N}_{S, \ell, \ell'}(t) \to0$ locally uniformly in $t$ in probability as $N\to \infty$. Similarly for   $\bar{M}^{N}_{R, \ell, \ell'} \to0$. 

We next prove the convergence of $\bar{M}^{N}_{\mfI, \ell, \ell'} (t,\mfa)$.
We apply Theorem \ref{thm-DD-conv0}. 
First, for each $t, \mfa \ge 0$, 
\begin{align*}
\E\left[\left(\bar{M}^{N}_{\mfI, \ell, \ell'} (t,\mfa)\right)^2 \right] & =  \frac{1}{N}  \nu^{I}_{\ell, \ell'} \E\int_{(t-\mfa)^+}^t \bar\mfI^N_{\ell}(s, \mfa-(t-s))  ds\,. 
\end{align*}
Observe that by \eqref{eqn-mfI-N-new}, for each $\ell$, and for each $t, \mfa \ge0$, 
\begin{align} \label{eqn-bar-mfI-n-bound}
\bar\mfI^N_{\ell}(t, \mfa) \le   \sum_{\ell'\in \sL}\Big(  \bar\mfI^N_{\ell'}(0,(\mfa-t)^+ ) + \bar{A}^N_{\ell'}(t) - \bar{A}^N_{\ell'}((t-\mfa)^+) \Big)\,. 
\end{align}
So  for $(t-\mfa)^+ < s< t$, we have
\[
\bar\mfI^N_{\ell}(s, \mfa-(t-s))  \le   \sum_{\ell'\in \sL}\Big(  \bar\mfI^N_{\ell'}(0,(\mfa-t)^+ ) + \bar{A}^N_{\ell'}(s) - \bar{A}^N_{\ell'}((t-\mfa)^+) \Big)\,. 
\]
Hence 
\begin{align*}
\E\left[\left(\bar{M}^{N}_{\mfI, \ell, \ell'} (t,\mfa)\right)^2 \right] 
& \le  \frac{1}{N}  \nu^{I}_{\ell, \ell'} \E\int_{(t-\mfa)^+}^t  \sum_{\ell'\in \sL}\Big(  \bar\mfI^N_{\ell'}(0,(\mfa-t)^+ ) + \bar{A}^N_{\ell'}(s) - \bar{A}^N_{\ell'}((t-\mfa)^+) \Big) ds \\
& \le   \frac{1}{N}  \nu^{I}_{\ell, \ell'} \mfa \, \E \sum_{\ell'\in \sL}\Big(  \bar\mfI^N_{\ell'}(0,(\mfa-t)^+ ) + \bar{A}^N_{\ell'}(t) - \bar{A}^N_{\ell'}((t-\mfa)^+) \Big) \\
&\le \frac{\nu^I_{\ell,\ell'}\mfa}{N}\\
& \to 0 \qasq N \to\infty\,.
\end{align*}
Thus by Markov's inequality, 
 for any $\ep>0$, 
 \[  \sup_{t \in [0,T]}\sup_{\mfa\in [0,T']} \P \big(  \big|\bar{M}^{N}_{\mfI, \ell, \ell'} (t,\mfa)\big|> \ep \big) \to 0\] 
 as $N\to\infty$.

Then, we check the two requirements of condition (ii) in Theorem \ref{thm-DD-conv0}. 
 For the first one, we show that for $\ep>0$, as $\delta\to 0$, 
\begin{align} \label{eqn-mfI-diff-c1}
\limsup_N \sup_{t \in [0, T]} \frac{1}{\delta} \P \left(\sup_{w\in [0,\delta]} \sup_{\mfa \in [0,T']} \big| \bar{M}^{N}_{\mfI, \ell, \ell'} (t+w,\mfa) - \bar{M}^{N}_{\mfI, \ell, \ell'} (t,\mfa)\big| > \epsilon \right) \to 0. 
\end{align} 
We have
\begin{align}\label{eqn-mfI-diff-p1}
& \Big|\bar{M}^{N}_{\mfI, \ell, \ell'} (t+w,\mfa) - \bar{M}^{N}_{\mfI, \ell, \ell'} (t,\mfa)\Big| \non \\
&\le\frac{1}{N} \Bigg|   \int_{t}^{t+w}\int_0^\infty{\bf1}_{u\le \nu^{I}_{\ell, \ell'}\mfI^N_{\ell}(s, \mfa-(t+w-s))}\bar{Q}^I_{\ell,\ell'}(ds,du)\Bigg|\non \\  
&\quad+\frac{1}{N} \Bigg|  \int_{(t-\mfa)^+}^{(t+w-\mfa)^+}\int_0^\infty{\bf1}_{u\le \nu^{I}_{\ell, \ell'}\mfI^N_{\ell}(s, \mfa-(t+w-s))}\bar{Q}^I_{\ell,\ell'}(ds,du)\Bigg|\non \\ 
&\quad+\frac{1}{N} \Bigg|  \int_{(t-\mfa)^+}^t\int_0^\infty{\bf1}_{ \nu^{I}_{\ell, \ell'}\mfI^N_{\ell}(s, \mfa-(t+w-s))<u\le \nu^{I}_{\ell, \ell'}\mfI^N_{\ell}(s, \mfa-(t-s))}\bar{Q}^I_{\ell,\ell'}(ds,du)\Bigg|\,.
\end{align}
The first term on the right of \eqref{eqn-mfI-diff-p1} satisfies
\begin{align*}
\frac{1}{N} \Bigg|\int_{t}^{t+w}\int_0^\infty{\bf1}_{u\le \nu^{I}_{\ell, \ell'}\mfI^N_{\ell}(s, \mfa-(t+w-s))}\bar{Q}^I_{\ell,\ell'}(ds,du)\Bigg|&\le\frac{1}{N}\int_{t}^{t+w}\int_0^\infty{\bf1}_{u\le \nu^{I}_{\ell, \ell'}\mfI^N_{\ell}(s, \mfa-(t+w-s))}Q^I_{\ell,\ell'}(ds,du)\\&\quad+\nu^{I}_{\ell, \ell'}\int_{t}^{t+w}\bar\mfI^N_{\ell}(s, \mfa-(t+w-s))ds\,.
\end{align*}
Hence, by the fact that $\mfI^N_{\ell}(s, \mfa)$ is increasing in $\mfa$ and $\mfI^N_{\ell}(s, \infty) \le 1$, we obtain 
\begin{align*}
\sup_{w\in[0,\delta], \mfa\in[0, T']}\frac{1}{N}& \Bigg|\int_{t}^{t+w}\int_0^\infty{\bf1}_{u\le \nu^{I}_{\ell, \ell'}\mfI^N_{\ell}(s, \mfa-(t+w-s))}\bar{Q}^I_{\ell,\ell'}(ds,du)\Bigg| \\
&\le\sup_{\mfa\in[0, T']}\frac{1}{N}\int_{t}^{t+\delta}\int_0^\infty{\bf1}_{u\le \nu^{I}_{\ell, \ell'}\mfI^N_{\ell}(s, \mfa-(t-s))}Q^I_{\ell,\ell'}(ds,du)+\nu^{I}_{\ell, \ell'}\delta\\
&\le \frac{1}{N}\left|\int_{t}^{t+\delta}\int_0^\infty{\bf1}_{u\le \nu^{I}_{\ell, \ell'}\mfI^N_{\ell}(s, \infty)}\bar{Q}^I_{\ell,\ell'}(ds,du)\right|+2\nu^{I}_{\ell, \ell'}\delta\,,
\end{align*}
For the first term on the right hand side, we have
\[
\E\bigg[\bigg(  \frac{1}{N}\left|\int_{t}^{t+\delta}\int_0^\infty{\bf1}_{u\le \nu^{I}_{\ell, \ell'}\mfI^N_{\ell}(s, \infty)}\bar{Q}^I_{\ell,\ell'}(ds,du)\right|\bigg)^2 \bigg] \le \frac{1}{N}  \nu^{I}_{\ell, \ell'} \delta. 
\]
Thus, 
\begin{align*}
\limsup_N &\sup_{t\in [0,T]} \frac{1}{\delta}\P\left(\sup_{w\in[0,\delta], \mfa\in[0, T']}\frac{1}{N} \Bigg|\int_{t}^{t+w}\int_0^\infty{\bf1}_{u\le \nu^{I}_{\ell, \ell'}\mfI^N_{\ell}(s, \mfa-(t+w-s))}\bar{Q}^I_{\ell,\ell'}(ds,du)\Bigg|>\epsilon\right)\non\\
&\le\epsilon^{-2}\limsup_N \frac{1}{\delta} \E\left(\sup_{w\in[0,\delta], \mfa\in[0, T']}\frac{1}{N^2}\left| \int_{t}^{t+w}\int_0^\infty{\bf1}_{u\le \nu^{I}_{\ell, \ell'}\mfI^N_{\ell}(s, \mfa-(t+w-s))}\bar{Q}^I_{\ell,\ell'}(ds,du)\right|^2\right)\non\\
& \le \epsilon^{-2}\limsup_N \frac{1}{\delta}  \Big( 2  \frac{1}{N}  \nu^{I}_{\ell, \ell'} \delta +   8(\nu^I_{\ell,\ell'})^2\delta^2 \Big) \non \\
&=   \epsilon^{-2} 8(\nu^I_{\ell,\ell'})^2\delta\,,
\end{align*}
which tends to $0$ as $\delta\to0$, as required by \eqref{eqn-mfI-diff-c1}.

For the second term on the right of \eqref{eqn-mfI-diff-p1},  we have
\begin{align*}
&\sup_{w\in[0,\delta],\mfa\in[0,T']}\frac{1}{N}\left|
\int_{(t-\mfa)^+}^{(t+w-\mfa)^+}\int_0^\infty {\bf1}_{u\le\nu^I_{\ell,\ell'}\mfI^N_\ell(s,\mfa-t-w+s)}\bar{Q}^I_{\ell,\ell'}(ds,du)\right|
\\&\le\sup_{w\in[0,\delta],\mfa\in[0,T']}\frac{1}{N}
\int_{(t-\mfa)^+}^{(t+w-\mfa)^+}\int_0^\infty {\bf1}_{u\le\nu^I_{\ell,\ell'}\mfI^N_\ell(s,\mfa-t-w+s)}{Q}^I_{\ell,\ell'}(ds,du)\\
&\quad+\sup_{w\in[0,\delta],\mfa\in[0,T']}\int_{(t-\mfa)^+}^{(t+w-\mfa)^+}\nu^I_{\ell,\ell'}\bar\mfI^N_\ell(s,\mfa-t-w+s)ds\\
&\le\sup_{\mfa\in[0,T']}\frac{1}{N}
\int_{(t-\mfa)^+}^{(t+\delta-\mfa)^+}\int_0^\infty {\bf1}_{u\le\nu^I_{\ell,\ell'}\mfI^N_\ell(s,\infty)}{Q}^I_{\ell,\ell'}(ds,du)\\
&\quad+\sup_{\mfa\in[0,T']}\int_{(t-\mfa)^+}^{(t+\delta-\mfa)^+}\nu^I_{\ell,\ell'}\bar\mfI^N_\ell(s,\infty)ds\\
&\le\sup_{\mfa\in[0,T']}\frac{1}{N}\left|\int_{(t-\mfa)^+}^{(t+\delta-\mfa)^+}\int_0^\infty {\bf1}_{u\le\nu^I_{\ell,\ell'}\mfI^N_\ell(s,\infty)}\bar{Q}^I_{\ell,\ell'}(ds,du)\right|\\
&\quad+2\sup_{\mfa\in[0,T']}\int_{(t-\mfa)^+}^{(t+\delta-\mfa)^+}\nu^I_{\ell,\ell'}\bar\mfI^N_\ell(s,\infty)ds
\end{align*}
The second term on the last right hand side is bounded by
$2\nu^I_{\ell,\ell'}\delta$, 
while the first term is bounded by
\begin{align*}
&\sup_{\mfa\in[0,T']}\frac{1}{N}\left|\int_0^{(t-\mfa)^+}\int_0^\infty {\bf1}_{u\le\nu^I_{\ell,\ell'}\mfI^N_\ell(s,+\infty)}\bar{Q}^I_{\ell,\ell'}(ds,du)\right|\\
&\quad+\sup_{\mfa\in[0,T']}\frac{1}{N}\left|\int_0^{(t+\delta-\mfa)^+}\int_0^\infty {\bf1}_{u\le\nu^I_{\ell,\ell'}\mfI^N_\ell(s,+\infty)}\bar{Q}^I_{\ell,\ell'}(ds,du)\right|\,.
\end{align*}
Each of the two terms in this sum is bounded by
\[\sup_{0\le r\le t+\delta}\frac{1}{N}\left|\int_0^{r}\int_0^\infty {\bf1}_{u\le\nu^I_{\ell,\ell'}\mfI^N_\ell(s,+\infty)}\bar{Q}^I_{\ell,\ell'}(ds,du)\right|,\]
which is the sup of a square integrable martingale. It follows from Doob's inequality that
\begin{align*}
\mathbb{E}\left( \sup_{0\le r\le t+\delta}\left|\frac{1}{N}\int_0^{r}\int_0^\infty {\bf1}_{u\le\nu^I_{\ell,\ell'}\mfI^N_\ell(s,+\infty)}\bar{Q}^I_{\ell,\ell'}(ds,du)\right|^2\right)
&\le\frac{4\nu^I_{\ell,\ell'}}{N}\mathbb{E}\int_0^{t+\delta}\bar\mfI^N_\ell(s,+\infty)ds\\
&\le\frac{4\nu^I_{\ell,\ell'}}{N}(t+\delta)\,.
\end{align*}
Thus we obtain
\begin{align*}
\limsup_N\sup_{t\in [0,T]}\frac{1}{\delta}\P&\left(\sup_{w\in[0,\delta],\mfa\in[0,T']}\frac{1}{N}\left|
\int_{(t-\mfa)^+}^{(t+w-\mfa)^+}\int_0^\infty {\bf1}_{u\le\nu^I_{\ell,\ell'}\mfI^N_\ell(s,\mfa-t-w+s)}\bar{Q}^I_{\ell,\ell'}(ds,du)\right|>\epsilon\right)\non\\
& \le \epsilon^{-2}\limsup_N \frac{1}{\delta}  \Big( \frac{16\nu^I_{\ell,\ell'}}{N}(T+\delta) +   8(\nu^I_{\ell,\ell'})^2\delta^2 \Big) \non \\
&=   \epsilon^{-2} 8(\nu^I_{\ell,\ell'})^2\delta\,,
\end{align*}
which converges to 0 as $\delta \to 0$, as required by \eqref{eqn-mfI-diff-c1}.

We finally consider the third term  on the right of \eqref{eqn-mfI-diff-p1}. 
We have
\begin{align} \label{eqn-mfI-diff-p1-3} 
\sup_{w\in[0,\delta], \mfa\in[0, T']}\frac{1}{N}& \Bigg|  \int_{(t-\mfa)^+}^t\int_0^\infty{\bf1}_{ \nu^{I}_{\ell, \ell'}\mfI^N_{\ell}(s, \mfa-(t+w-s))<u\le \nu^{I}_{\ell, \ell'}\mfI^N_{\ell}(s, \mfa-(t-s))}\bar{Q}^I_{\ell,\ell'}(ds,du)\Bigg| \nonumber\\
&\le\sup_{\mfa\in[0, T']}\frac{1}{N}  \int_0^t\int_0^\infty{\bf1}_{ \nu^{I}_{\ell, \ell'}\mfI^N_{\ell}(s, \mfa-(t+\delta-s))<u\le \nu^{I}_{\ell, \ell'}\mfI^N_{\ell}(s, \mfa-(t-s))}Q^I_{\ell,\ell'}(ds,du) \nonumber\\
&\quad+\nu^{I}_{\ell, \ell'}\sup_{\mfa\in[0, T']}\int_0^t\left(\bar\mfI^N_{\ell}(s, \mfa-(t-s))-\bar\mfI^N_{\ell}(s, \mfa-(t+\delta-s))\right)ds \nonumber\\
&\le\sup_{\mfa\in[0, T']}\frac{1}{N}  \left|\int_0^t\int_0^\infty{\bf1}_{ \nu^{I}_{\ell, \ell'}\mfI^N_{\ell}(s, \mfa-(t+\delta-s))<u\le \nu^{I}_{\ell, \ell'}\mfI^N_{\ell}(s, \mfa-(t-s))}\bar{Q}^I_{\ell,\ell'}(ds,du)\right| \nonumber\\
&\quad+2\nu^{I}_{\ell, \ell'}\sup_{\mfa\in[0, T']}\int_0^t\left(\bar\mfI^N_{\ell}(s, \mfa-(t-s))-\bar\mfI^N_{\ell}(s, \mfa-(t+\delta-s))\right)ds\,.
\end{align}

Let us consider the second term. For that sake, we first upper bound the integrand in the $ds$ integral for each fixed
$s$ and $\mfa$. If  $\mfa > t+\delta$, then
\begin{align*}
\bar\mfI^N_{\ell}(s, \mfa-(t-s))-\bar\mfI^N_{\ell}(s, \mfa-(t+\delta-s))\le
\sum_{\ell'}\left(\bar\mfI^N_{\ell'}(0,(\mfa-t)^+)-\bar\mfI^N_{\ell'}(0,(\mfa-t-\delta)^+\right)\,. 
\end{align*}
 If $\mfa<t$, then
 \begin{align*}
 \bar\mfI^N_{\ell}(s, \mfa-(t-s))-\bar\mfI^N_{\ell}(s, \mfa-(t+\delta-s))\le
\sum_{\ell'}\left(\bar{A}^N_{\ell'}((t+\delta-\mfa)^+)- \bar{A}^N_{\ell'}((t-\mfa)^+)\right)\,.
 \end{align*}
Finally, if $t<\mfa<t+\delta$, then
\begin{align*}
 \bar\mfI^N_{\ell}(s, \mfa-(t-s))-\bar\mfI^N_{\ell}(s, \mfa-(t+\delta-s))&\le
\sum_{\ell'}\left(\bar\mfI^N_{\ell'}(0,(\mfa-t)^+)+\bar{A}^N_{\ell'}((t+\delta-\mfa)^+)\right),
 \end{align*}
 which is upper bounded by the sum of the two above right hand sides.
Finally, the second term in the above upper bound is bounded from above by 
\begin{align} \label{eqn-mfI-diff-p31}
2\nu^{I}_{\ell, \ell'}t
\sup_{\mfa\in[0, T']} &\Bigg\{\sum_{\ell'}\left(\bar\mfI^N_{\ell'}(0,(\mfa-t)^+)-\bar\mfI^N_{\ell'}(0,(\mfa-t-\delta)^+\right) \nonumber\\
&\quad +
\sum_{\ell'}\left(\bar{A}^N_{\ell'}((t+\delta-\mfa)^+)- \bar{A}^N_{\ell'}((t-\mfa)^+)\right)\Bigg\}\,.
\end{align}
\medskip

We first note that, from Assumption 2.1, for any $p\ge1$, $\ell'$, $\epsilon>0$,
\begin{align}\label{eqn-mfI-diff-p32}
\limsup_{N\to\infty} \sup_{t\in [0,T]} \mathbb{P}& \left(\sup_{\mfa\in[0, T']}\left[\bar\mfI^N_{\ell'}(0,(\mfa-t)^+)-\bar\mfI^N_{\ell'}(0,(\mfa-t-\delta)^+\right]>\epsilon\right) \nonumber \\&\le\mathbb{P} \left(\sup_{\mfa\in[0, T']}\left[\bar\mfI_{\ell'}(0,(\mfa-t)^+)-\bar\mfI_{\ell'}(0,(\mfa-t-\delta)^+\right]>\epsilon\right) \nonumber\\
&\le\frac{C^p}{\epsilon^p}\delta^{\alpha p},
\end{align}
so that it suffices to choose $p>\alpha^{-1}$ in order for the last upper bound to be of the form $C\delta^\beta$, with $\beta>1$.

Next, thanks to Lemma 4.1,
\begin{align}\label{eqn-mfI-diff-p33}
\limsup_{N\to\infty} \sup_{t\in [0,T]} \mathbb{P}& \left(\sup_{\mfa\in[0, T']}\left[\bar{A}^N_{\ell'}((t+\delta-\mfa)^+)- \bar{A}^N_{\ell'}((t-\mfa)^+)\right]>\epsilon\right) \nonumber \\&\le\mathbb{P}\left(\sup_{\mfa\in[0, T']}\left[\bar{A}_{\ell'}((t+\delta-\mfa)^+)- \bar{A}_{\ell'}((t-\mfa)^+)\right]>\epsilon\right)\nonumber\\
&=0,
\end{align}
as soon as $\lambda^\ast\beta^\ast\delta<\epsilon$.

For the first term on the right hand side of \eqref{eqn-mfI-diff-p1-3}, we observe that it is bounded by 
\begin{align*}
& \frac{1}{N} \bigg| \int_0^t\int_0^\infty{\bf1}_{ u \le \nu^{I}_{\ell, \ell'}\mfI^N_{\ell}(s, \mfa-(t-s))}\bar{Q}^I_{\ell,\ell'}(ds,du) \bigg|  + \frac{1}{N} \bigg|  \int_0^t\int_0^\infty{\bf1}_{ u \le \nu^{I}_{\ell, \ell'}\mfI^N_{\ell}(s, \mfa-(t+\delta-s)) }\bar{Q}^I_{\ell,\ell'}(ds,du) \bigg|\,.
\end{align*}
Let
\[ M^N(\mfa)=\frac{1}{N}\int_0^t\int_0^\infty {\bf1}_{u\le\nu^I_{\ell,\ell'}\mfI^N_\ell(s,\mfa-t+s)}\bar{Q}^I_{\ell,\ell'}(ds,du)\,.
\]
It then suffices to show that $\sup_{\mfa\in[0,T']}|M^N(\mfa)|\to0$ in probability, as $N\to\infty$, for each fixed $0<t \le T$. 

We first note that
\[ \mathbb{E} \left[|M^N(\mfa)|^2\right]\le\frac{\nu^I_{\ell,\ell'}t}{N}\to0, \ \text{ as }N\to\infty\,.\]
Let $\mfa'>\mfa$. We have
\begin{align*}
M^N(\mfa')-M^N(\mfa)&=\frac{1}{N}\int_0^t\int_0^\infty 
{\bf1}_{\nu^I_{\ell,\ell'}\mfI^N_\ell(s,\mfa-t+s)<u\le\nu^I_{\ell,\ell'}\mfI^N_\ell(s,\mfa'-t+s)}\bar{Q}^I_{\ell,\ell'}(ds,du),\\
|M^N(\mfa')-M^N(\mfa)|&\le\frac{1}{N}\int_0^t\int_0^\infty 
{\bf1}_{\nu^I_{\ell,\ell'}\mfI^N_\ell(s,\mfa-t+s)<u\le\nu^I_{\ell,\ell'}\mfI^N_\ell(s,\mfa'-t+s)}{Q}^I_{\ell,\ell'}(ds,du)\\
&\quad+\nu^I_{\ell,\ell'}\int_0^t\left[\bar\mfI^N_\ell(s,\mfa'-t+s)-\bar\mfI^N_\ell(s,\mfa-t+s)\right]ds
\end{align*}
Since the last right hand side is increasing in $\mfa'$, for any $\rho>0$,
\begin{align} \label{eqn-Ma-inc-bound}
\sup_{\mfa<\mfa'\le\mfa+\rho}|M^N(\mfa')-M^N(\mfa)|&\le\frac{1}{N}\int_0^t\int_0^\infty 
{\bf1}_{\nu^I_{\ell,\ell'}\mfI^N_\ell(s,\mfa-t+s)<u\le\nu^I_{\ell,\ell'}\mfI^N_\ell(s,\mfa+\rho-t+s)}{Q}^I_{\ell,\ell'}(ds,du) \nonumber\\
&\quad+\nu^I_{\ell,\ell'}\int_0^t\left[\bar\mfI^N_\ell(s,\mfa+\rho-t+s)-\bar\mfI^N_\ell(s,\mfa-t+s)\right]ds\nonumber\\
&=\frac{1}{N}\int_0^t\int_0^\infty 
{\bf1}_{\nu^I_{\ell,\ell'}\mfI^N_\ell(s,\mfa-t+s)<u\le\nu^I_{\ell,\ell'}\mfI^N_\ell(s,\mfa+\rho-t+s)}\bar{Q}^I_{\ell,\ell'}(ds,du)\nonumber\\
&\quad+2\nu^I_{\ell,\ell'}\int_0^t\left[\bar\mfI^N_\ell(s,\mfa+\rho-t+s)-\bar\mfI^N_\ell(s,\mfa-t+s)\right]ds\,.
\end{align}
We have
\begin{align*}
\mathbb{E}&\left[\left|\frac{1}{N}\int_0^t\int_0^\infty 
{\bf1}_{\nu^I_{\ell,\ell'}\mfI^N_\ell(s,\mfa-t+s)<u\le\nu^I_{\ell,\ell'}\mfI^N_\ell(s,\mfa+\rho-t+s)}\bar{Q}^I_{\ell,\ell'}(ds,du)\right|^2\right]\\
&\le\frac{\nu^I_{\ell,\ell'}}{N}\mathbb{E} \bigg[\int_0^t\left[\bar\mfI^N_\ell(s,\mfa+\rho-t+s)-\bar\mfI^N_\ell(s,\mfa-t+s)\right]ds \bigg]\\
&\le\frac{\nu^I_{\ell,\ell'}}{N}Ct[\rho^\alpha+\rho],
\end{align*}
for some constant $C$, 
where the last inequality follows from a similar argument for the bound in \eqref{eqn-mfI-diff-p31} and then by Assumption \ref{AS-initial} and \eqref{eqn-barA-inc} in Lemma \ref{lem-An-conv}, 
and the second term on the right  satisfies (we choose $p>1/\alpha$):
\[\mathbb{E}\left(\left|2\nu^I_{\ell,\ell'}\int_0^t\left[\bar\mfI^N_\ell(s,\mfa+\rho-t+s)-\bar\mfI^N_\ell(s,\mfa-t+s)\right]ds\right|^p\right)
\le Ct^p[\rho^{\alpha p}+\rho^p]\,.\]
Finally, for any $\epsilon,\eta>0$, 
\begin{align*}
\mathbb{P}\left(\sup_{\mfa<\mfa'\le\mfa+\rho}|M^N(\mfa')-M^N(\mfa)|\ge\epsilon\right)&\le C_T\left(\frac{4(\rho^\alpha+\rho)}{N\epsilon^2}+
\frac{2^p(\rho^{\alpha p}+\rho^p)}{\epsilon^p}\right)\\
&\le C_T\left(\frac{\rho^\alpha}{N\epsilon^2}+\frac{\rho^{\alpha p}}{\epsilon^p}\right) \,,
\end{align*}
for some constant $C_T>0$, 
since $\alpha\le 1$, and we shall choose below $\rho\le 1$.
Consequently
\begin{align*}
\frac{1}{\rho}\mathbb{P}\left(\sup_{\mfa<\mfa'\le\mfa+\rho}|M^N(\mfa')-M^N(\mfa)|\ge\epsilon\right)\le \eta\, ,
\end{align*}
if $\rho=\left(\frac{\eta\epsilon^p}{2C_T}\right)^{\frac{1}{\alpha p-1}}$, and 
$N\ge N_0=\left(\frac{2C_T}{\eta}\right)^{\frac{\alpha(p-1)}{\alpha p-1}}\times\epsilon^{-2-p\frac{1-\alpha}{\alpha p-1}}$.

It follows from this and the Corollary on page 83 in Billingsley  \cite{billingsley1999convergence}  that for any $\epsilon,\ \eta>0$, there exists $\rho>0$ and $N_0$ such that for any $N\ge N_0$,
\begin{align}\label{modcont}
\mathbb{P}\left(\sup_{0\le\mfa\le \mfa'\le T',\ \mfa'-\mfa\le\rho}|M^N(\mfa)-M^N(\mfa')|\ge\epsilon\right)\le\eta\,.
\end{align}

Now we are in a position to prove  that  
\begin{align*}
\sup_{\mfa\in[0,T']}|M^N(a)|\to0,\ \text{ in probability, as }N\to\infty\, ,
\end{align*}
i.e., that for any $\epsilon,\ \eta>0$, there exists $N_{\epsilon,\eta}$ such that for any $N\ge N_{\epsilon,\eta}$,
\begin{align}\label{convunif}
\mathbb{P}\left(\sup_{\mfa\in[0,T']}|M^N(a)|\ge \epsilon\right)\le\eta\,.
\end{align}

From \eqref{modcont}, we can first choose $\rho$ and $N_0$ such that 
\begin{align}\label{modcont2}
\mathbb{P}\left(\sup_{0\le\mfa\le \mfa'\le T',\ \mfa'-\mfa\le\rho}|M^N(\mfa)-M^N(\mfa')|\ge\frac{\epsilon}{2}\right)\le\frac{\eta}{2}\,.
\end{align}

Next we consider the following finite number of sequences indexed by $N$: $\{M^N(i\rho\wedge T'), \ 0\le i\le (T'/\rho)+1\}$. Since for each   $0\le i\le (T'/\rho)+1$, $M^N(i\rho\wedge T')\to0$ in probability, as $N\to\infty$,
\begin{align*}
\sup_{0\le i\le (T'/\rho)+1}|M^N(i\rho\wedge T')|\to0\ \text{ in probability, as }N\to\infty\,.
\end{align*}
Consequently, there exists $N_{\epsilon,\eta}\ge N_0$ such that
\begin{align}\label{convfin}
\mathbb{P}\left(\sup_{0\le i\le (T'/\rho)+1}|M^N(i\rho\wedge T')|\ge\frac{\epsilon}{2}\right)\le\frac{\eta}{2}\,.
\end{align}
Now \eqref{convunif} follows from \eqref{modcont2} and \eqref{convfin}, since clearly
\begin{align*}
\sup_{\mfa\in[0,T']}|M^N(a)|\le\sup_{0\le i\le (T'/\rho)+1}|M^N(i\rho\wedge T')|+\sup_{0\le\mfa\le \mfa'\le T',\ \mfa'-\mfa\le\rho}|M^N(\mfa)-M^N(\mfa')|\,.
\end{align*}

Therefore, combining the above, we obtain 
\begin{align*}
\limsup_N\sup_{t\in [0,T]}&\frac{1}{\delta}
\P\left(\sup_{w\in[0,\delta], \mfa\in[0, T']}\frac{1}{N} \Bigg|  \int_{(t-\mfa)^+}^t\int_0^\infty{\bf1}_{ \nu^{I}_{\ell, \ell'}\mfI^N_{\ell}(s, \mfa-(t+w-s))<u\le \nu^{I}_{\ell, \ell'}\mfI^N_{\ell}(s, \mfa-(t-s))}\bar{Q}^I_{\ell,\ell'}(ds,du)\Bigg|>\epsilon\right) \\
& \quad \to 0 \quad \text{as} \,\, \delta\to 0. 
\end{align*}
We have thus shown that \eqref{eqn-mfI-diff-c1} holds.

\medskip

 For the second requirement of condition (ii) in Theorem \ref{thm-DD-conv0}, we show that for $\ep>0$, as $\delta\to 0$, 
\begin{align} \label{eqn-mfI-diff-c2}
\limsup_N \sup_{t \in [0, T]} \frac{1}{\delta} \P \left(\sup_{v\in [0,\delta]} \sup_{t \in [0,T]} \big| \bar{M}^{N}_{\mfI, \ell, \ell'} (t,\mfa+v) - \bar{M}^{N}_{\mfI, \ell, \ell'} (t,\mfa)\big| > \varepsilon\right) \to 0. 
\end{align} 
We have 
\begin{align*}
& \Big|\bar{M}^{N}_{\mfI, \ell, \ell'} (t,\mfa+v) - \bar{M}^{N}_{\mfI, \ell, \ell'} (t,\mfa)\Big| \\
&= \Bigg|  \frac{1}{N} \left(\int_{(t-\mfa-v)^+}^{t}\int_0^\infty{\bf1}_{u\le \nu^{I}_{\ell, \ell'}\mfI^N_{\ell}(s, \mfa+v-(t-s))}Q^I_{\ell,\ell'}(ds,du)   -  \nu^{I}_{\ell, \ell'} \int_{(t-\mfa-v)^+}^{t} \mfI^N_{\ell'}(s, \mfa+v-(t-s))  ds   \right) \non\\
& \quad -  \frac{1}{N} \left(\int_{(t-\mfa)^+}^t\int_0^\infty{\bf1}_{u\le \nu^{I}_{\ell, \ell'}\mfI^N_{\ell}(s, \mfa-(t-s))}Q^I_{\ell,\ell'}(ds,du)   -  \nu^{I}_{\ell, \ell'} \int_{(t-\mfa)^+}^t \mfI^N_{\ell'}(s, \mfa-(t-s))  ds   \right) \Bigg| \non\\
&\le  \frac{1}{N} \int_{(t-\mfa-v)^+}^{t}\int_0^\infty{\bf1}_{ \nu^{I}_{\ell, \ell'}  \mfI^N_{\ell}(s, \mfa -(t-s)) < u\le \nu^{I}_{\ell, \ell'}\mfI^N_{\ell}(s, \mfa+v-(t-s))}Q^I_{\ell,\ell'}(ds,du)   \\
& \quad + \frac{1}{N} \int_{(t-\mfa-v)^+}^{(t-\mfa)^+} \int_0^\infty{\bf1}_{u\le \nu^{I}_{\ell, \ell'}\mfI^N_{\ell}(s, \mfa-(t-s))}Q^I_{\ell,\ell'}(ds,du)    \\
& \quad +  \nu^{I}_{\ell, \ell'} \int_{(t-\mfa-v)^+}^{(t-\mfa)^+} \bar\mfI^N_{\ell'}(s, \mfa+v-(t-s))  ds\,.  \non
\end{align*} 
Clearly, the same arguments used to verify condition (i) allow us to conclude condition (ii) of Theorem \ref{thm-DD-conv0}.
\end{proof} 

We next prove the convergence of the processes $\bar\mfI^{N,0}_\ell(t, \mfa)$ and $\bar\mfI^{N,1}_\ell(t, \mfa)$.
We will only provide the detailed proof for the convergence of $\bar\mfI^{N,1}_\ell(t, \mfa)$ since the proof of that of $\bar\mfI^{N,0}_\ell(t, \mfa)$ follows the same steps with some modifications. 
 
\begin{lemma} \label{lem-bar-mfI0-conv}
Under Assumptions \ref{AS-initial} and \ref{AS-lambda}, for each $\ell\in \sL$, 
\begin{align}
\bar\mfI^{N,0}_\ell(t, \mfa)  \to \bar\mfI^{0}_\ell(t, \mfa) 
\end{align} 
in probability, uniformly in $t$ and $\mfa$, as $N\to\infty$, where
\begin{align} \label{eqn-bar-mfI0}
\bar\mfI^{0}_\ell(t, \mfa) &= \sum_{\ell'=1}^L  \int_0^{(\mfa-t)^+} \bigg( \int_0^t p_{\ell',\ell}(u)   F_0(du|y) \bigg) \bar\mfI_{\ell'}(0,dy)\,. 
\end{align}
\end{lemma}

\begin{lemma} \label{lem-bar-mfI1-conv}
Under Assumptions \ref{AS-initial} and \ref{AS-lambda}, for each $\ell\in \sL$,  along a convergent subsequence of $\bar{A}^N_\ell$ with limit $\bar{A}_\ell$, as $N\to\infty$, 
\begin{align}
\bar\mfI^{N,1}_\ell(t, \mfa)  \Rightarrow \bar\mfI^{1}_\ell(t, \mfa) 
\end{align} 
 for the topology of locally uniform convergence in $t$ and $\mfa$, 
where
\begin{align} \label{eqn-bar-mfI0}
\bar\mfI^{1}_\ell(t, \mfa) &= \sum_{\ell'=1}^L \int_{(t-\mfa)^+}^t \int_0^{t-s} p_{\ell',\ell}(u) F(du)  d \bar{A}_{\ell'}(s) \,. 
\end{align}
In fact we have the joint convergence $(\bar{A}^N_\ell(t),\bar\mfI^{N,1}_\ell(t, \mfa))\Rightarrow(\bar{A}_\ell(t),\bar\mfI^{1}_\ell(t, \mfa))$,
for the topology of locally uniform convergence in $t$ and $\mfa$.
\end{lemma}

\begin{proof}
Define
\begin{align} \label{eqn-breve-mfI-N-1}
\breve\mfI^{N,1}_\ell(t, \mfa) &:= \frac{1}{N} \sum_{\ell'=1}^L \sum_{i=A^N_{\ell'}((t-\mfa)^+)+1}^{A^N_{\ell'}(t)} 
 \int_0^{t-\tau_{i}^{\ell',N}} p_{\ell',\ell}(u) F(du)  \non\\
 & =  \sum_{\ell'=1}^L \int_{(t-\mfa)^+}^t  \int_0^{t-s} p_{\ell',\ell}(u) F(du)  d \bar{A}^N_{\ell'}(s) \,. 
\end{align} 
(Here the integral $\int_a^b$ stands for $\int_{(a,b]}$.)
By Lemma \ref{le:Portmanteau}, for each $t, \mfa \ge 0$,  
\begin{equation} \label{eqn-breve-mfI-conv}
\breve\mfI^{N,1}_\ell(t, \mfa)  \Rightarrow \bar\mfI^{1}_\ell(t, \mfa) \qasq N \to \infty. 
\end{equation} 
Then to show that the convergence $\breve\mfI^{N,1}_\ell(t, \mfa)  \Rightarrow \bar\mfI^{1}_\ell(t, \mfa)$ holds locally uniformly in $t$ and $\mfa$, it suffices to show that for any $\varepsilon>0$, there exists $\delta>0$ such that for any $t, \mfa\ge 0$, 
\begin{align}
\limsup_N \P \left( \sup_{t \le t' \le t+\delta, \, \mfa \le \mfa' \le \mfa+\delta} \Big| \breve\mfI^{N,1}_\ell(t, \mfa) - \breve\mfI^{N,1}_\ell(t', \mfa')   \Big| >\varepsilon\right) =0. 
\end{align}
This follows from the second  representation in \eqref{eqn-breve-mfI-N-1}, and the convergence of $\bar{A}^N_\ell$ in Lemma \ref{lem-An-conv}.

Next we consider the difference 
\begin{align*}
V^N(t,\mfa) &  := \bar\mfI^{N,1}_\ell(t, \mfa)   - \breve\mfI^{N,1}_\ell(t, \mfa)  \\
&= \frac{1}{N} \sum_{\ell'=1}^L \sum_{i=A^N_{\ell'}((t-\mfa)^+)+1}^{A^N_{\ell'}(t)} \bigg({\bf1}_{\tau_{i}^{\ell',N}+\eta^{\ell'}_i \le t}  {\bf1}_{X_i^{\ell'}(\eta^{\ell'}_i) = \ell} -  \int_0^{t-\tau_{i}^{\ell',N}} p_{\ell',\ell}(u) F(du) \bigg)\,. 
\end{align*}
We apply Theorem \ref{thm-DD-conv0} to show that 
\[
V^N(t,\mfa) \to 0 
\]
in probability in the topology of locally uniform convergence in $t$ and $\mfa$ as $N \to \infty$.
For condition (i) in Theorem \ref{thm-DD-conv0}, we have
\begin{align*}
\P(V^N(t,\mfa) >\epsilon)  & \le \frac{1}{\epsilon^2} \E\big[V^N(t,\mfa)^2\big ] \\
& \le  \frac{L}{\epsilon^2N}  \sum_{\ell'=1}^L \E\left[  \int_{(t-\mfa)^+}^t  \int_0^{t-s} p_{\ell',\ell}(u) F(du)  \left(1-  \int_0^{t-s} p_{\ell',\ell}(u) F(du) \right) d \bar{A}^N_{\ell'}(s) \right] \\
& \le \frac{L}{\epsilon^2N}  \sum_{\ell'=1}^L \E\left[  \int_{(t-\mfa)^+}^t  \int_0^{t-s} p_{\ell',\ell}(u) F(du)  \bar{\Upsilon}^N_{\ell'}(s) ds \right] \\
& \le \frac{L}{\epsilon^2N}  \lambda^* \beta^*  \sum_{\ell'=1}^L  \int_{(t-\mfa)^+}^t  \int_0^{t-s} p_{\ell',\ell}(u) F(du)  ds\,,
\end{align*}
and thus,
\begin{align*}
\sup_{t\in [0,T],\, \mfa \in [0, T']}\P(V^N(t,\mfa) >\epsilon) \to 0 \qasq N \to \infty. 
\end{align*} 
 
We then check the tightness requirements in condition (ii) of Theorem \ref{thm-DD-conv0}. 
For the first, we show that for $\ep>0$, as $\delta\to 0$, 
\begin{align} \label{eqn-V-diff-c1}
\limsup_N \sup_{t \in [0, T]} \frac{1}{\delta} \P \left(\sup_{u\in [0,\delta]} \sup_{\mfa \in [0,T']} \big| V^N(t+u,\mfa) - V^N(t,\mfa)\big| \right) \to 0. 
\end{align} 
We have 
\begin{align*}
 & \big| V^N(t+u,\mfa) - V^N(t,\mfa)\big|  \\
 & = \Bigg|  \frac{1}{N} \sum_{\ell'=1}^L \sum_{i=A^N_{\ell'}((t+u-\mfa)^+)+1}^{A^N_{\ell'}(t+u)} \bigg({\bf1}_{\tau_{i}^{\ell',N}+\eta^{\ell'}_i \le t+u}  {\bf1}_{X_i^{\ell'}(\eta^{\ell'}_i) = \ell} -  \int_0^{t+u-\tau_{i}^{\ell',N}} p_{\ell',\ell}(r) F(dr) \bigg)\\
 & \quad -  \frac{1}{N} \sum_{\ell'=1}^L \sum_{i=A^N_{\ell'}((t-\mfa)^+)+1}^{A^N_{\ell'}(t)} \bigg({\bf1}_{\tau_{i}^{\ell',N}+\eta^{\ell'}_i \le t}  {\bf1}_{X_i^{\ell'}(\eta^{\ell'}_i) = \ell} -  \int_0^{t-\tau_{i}^{\ell',N}} p_{\ell',\ell}(r) F(dr)  \bigg)\Bigg| \\
 & = \Bigg|  \frac{1}{N} \sum_{\ell'=1}^L \sum_{i=A^N_{\ell'}((t-\mfa)^+)+1}^{A^N_{\ell'}(t+u)}\bigg[ \bigg({\bf1}_{\tau_{i}^{\ell',N}+\eta^{\ell'}_i \le t+u}  {\bf1}_{X_i^{\ell'}(\eta^{\ell'}_i) = \ell} -  \int_0^{t+u-\tau_{i}^{\ell',N}} p_{\ell',\ell}(r) F(dr) \bigg)\\
 & \qquad  \qquad \qquad \qquad \qquad -  \bigg({\bf1}_{\tau_{i}^{\ell',N}+\eta^{\ell'}_i \le t}  {\bf1}_{X_i^{\ell'}(\eta^{\ell'}_i) = \ell} -  \int_0^{t-\tau_{i}^{\ell',N}} p_{\ell',\ell}(r) F(dr)  \bigg) \bigg] \\
 & \qquad - \frac{1}{N} \sum_{\ell'=1}^L \sum_{i=A^N_{\ell'}((t-\mfa)^+)+1}^{A^N_{\ell'}((t+u-\mfa)^+)} \bigg({\bf1}_{\tau_{i}^{\ell',N}+\eta^{\ell'}_i \le t+u}  {\bf1}_{X_i^{\ell'}(\eta^{\ell'}_i) = \ell} -  \int_0^{t+u-\tau_{i}^{\ell',N}} p_{\ell',\ell}(r) F(dr)  \bigg)  \\
 & \quad +  \frac{1}{N} \sum_{\ell'=1}^L \sum_{i=A^N_{\ell'}(t)+1}^{A^N_{\ell'}(t+u)} \bigg({\bf1}_{\tau_{i}^{\ell',N}+\eta^{\ell'}_i \le t}  {\bf1}_{X_i^{\ell'}(\eta^{\ell'}_i) = \ell} -  \int_0^{t-\tau_{i}^{\ell',N}} p_{\ell',\ell}(r) F(dr) \bigg)\Bigg| \\
 & \le    \frac{1}{N} \sum_{\ell'=1}^L \sum_{i=A^N_{\ell'}((t-\mfa)^+)+1}^{A^N_{\ell'}(t+u)} \bigg|{\bf1}_{t< \tau_{i}^{\ell',N}+\eta^{\ell'}_i \le t+u}  {\bf1}_{X_i^{\ell'}(\eta^{\ell'}_i) = \ell} -  \int_{t-\tau_{i}^{\ell',N}}^{t+u-\tau_{i}^{\ell',N}} p_{\ell',\ell}(r) F(dr)  \bigg|\\
 & \qquad + \frac{1}{N} \sum_{\ell'=1}^L \sum_{i=A^N_{\ell'}((t-\mfa)^+)+1}^{A^N_{\ell'}((t+u-\mfa)^+)} \bigg|{\bf1}_{\tau_{i}^{\ell',N}+\eta^{\ell'}_i \le t+u}  {\bf1}_{X_i^{\ell'}(\eta^{\ell'}_i) = \ell} -  \int_0^{t+u-\tau_{i}^{\ell',N}} p_{\ell',\ell}(r) F(dr)  \bigg| \\
 & \quad +  \frac{1}{N} \sum_{\ell'=1}^L \sum_{i=A^N_{\ell'}(t)+1}^{A^N_{\ell'}(t+u)} \bigg|{\bf1}_{\tau_{i}^{\ell',N}+\eta^{\ell'}_i \le t}  {\bf1}_{X_i^{\ell'}(\eta^{\ell'}_i) = \ell} -  \int_0^{t-\tau_{i}^{\ell',N}} p_{\ell',\ell}(r) F(dr)  \bigg|  \\
 & \le  \frac{1}{N} \sum_{\ell'=1}^L \sum_{i=A^N_{\ell'}((t-\mfa)^+)+1}^{A^N_{\ell'}(t+u)}  {\bf1}_{t< \tau_{i}^{\ell',N}+\eta^{\ell'}_i \le t+u}  {\bf1}_{X_i^{\ell'}(\eta^{\ell'}_i) = \ell}  \\
 & \qquad +  \frac{1}{N} \sum_{\ell'=1}^L \sum_{i=A^N_{\ell'}((t-\mfa)^+)+1}^{A^N_{\ell'}(t+u)}  \int_{t-\tau_{i}^{\ell',N}}^{t+u-\tau_{i}^{\ell',N}} p_{\ell',\ell}(r) F(dr)  \\
 & \qquad + \sum_{\ell'=1}^L \Big(\bar{A}^N_{\ell'}(t+u-\mfa) -\bar{A}^N_{\ell'}((t-\mfa)^+)  \Big) 
 +  \sum_{\ell'=1}^L \Big(\bar{A}^N_{\ell'}(t+u) -\bar{A}^N_{\ell'}(t)  \Big) \,.
\end{align*} 
Thus
\begin{align} \label{eqn-V-diff-p1}
& \P \left(\sup_{u\in [0,\delta]} \sup_{\mfa \in [0,T']} \big| V^N(t+u,\mfa) - V^N(t,\mfa)\big| > \epsilon \right)  \non \\
& \le \P \left(   \frac{1}{N} \sum_{\ell'=1}^L \sum_{i=A^N_{\ell'}((t-T')^+)+1}^{A^N_{\ell'}(t+\delta)}  {\bf1}_{t< \tau_{i}^{\ell',N}+\eta^{\ell'}_i \le t+\delta}  {\bf1}_{X_i^{\ell'}(\eta^{\ell'}_i) = \ell}> \epsilon/3 \right)  \non \\
& \quad +  \P \left( \frac{1}{N} \sum_{\ell'=1}^L \sum_{i=A^N_{\ell'}((t-T')^+)+1}^{A^N_{\ell'}(t+\delta)}  \int_{t-\tau_{i}^{\ell',N}}^{t+\delta-\tau_{i}^{\ell',N}} p_{\ell',\ell}(r) F(dr)  > \epsilon/3 \right)  \non \\
& \quad +  2 \P \left( \sup_{0 \le t \le T}  \sum_{\ell'=1}^L \Big|\bar{A}^N_{\ell'}(t+\delta) -\bar{A}^N_{\ell'}(t)  \Big| > \epsilon/6 \right)\,.
\end{align} 
For the first term, let $\{\widetilde{Q}_\ell(ds,  du, dr, d\theta),\ 1\le \ell\le L\}$ denote a collection of i.i.d. PRM on $\R_+^3\times \sL$ with mean measure $ds \times du \times F(dr) \times \mu_\ell(r, d\theta)$, where for each $r>0$, $\mu_\ell(r, \{\ell'\}) = p_{\ell,\ell'}(r)$. We denote by
 $\overline{Q}_\ell(ds,  du, dr, d\theta)$ be the compensated PRM associated to $\widetilde{Q}_\ell$, $1\le \ell\le L$.  We have 
\[
\sum_{i=A^N_{\ell'}((t-T')^+)+1}^{A^N_{\ell'}(t+\delta)}  {\bf1}_{t< \tau_{i}^{\ell',N}+\eta^{\ell'}_i \le t+\delta}  {\bf1}_{X_i^{\ell'}(\eta^{\ell'}_i) = \ell} = \int_{(t-T')^+}^{t+\delta} \int_0^\infty \int_{t-s}^{t+\delta-s} \int_{\{\ell\}} {\bf1}_{u \le \Upsilon^N_\ell(s^-)} \widetilde{Q}_{\ell'}(ds,  du, dr, d\theta)\,.
\]
Thus, we have the first term 
\begin{align} \label{eqn-V-diff-p2}
& \P \left(   \frac{1}{N} \sum_{\ell'=1}^L \sum_{i=A^N_{\ell'}((t-T')^+)+1}^{A^N_{\ell'}(t+\delta)}  {\bf1}_{t< \tau_{i}^{\ell',N}+\eta^{\ell'}_i \le t+\delta}  {\bf1}_{X_i^{\ell'}(\eta^{\ell'}_i) = \ell}> \epsilon/3 \right) \non  \\
& \le 9 \epsilon^{-2} \E \left[ \left(   \frac{1}{N} \sum_{\ell'=1}^L  \int_{(t-T')^+}^{t+\delta} \int_0^\infty \int_{t-s}^{t+\delta-s} \int_{\{\ell\}} {\bf1}_{u \le \Upsilon^N_\ell(s^-)} \widetilde{Q}_{\ell'}(ds,  du, dr, d\theta) \right)^2   \right] \non  \\
& \le 18 \epsilon^{-2} \E \left[   \frac{1}{N^2} \sum_{\ell'=1}^L  \left(  \int_{(t-T')^+}^{t+\delta} \int_0^\infty \int_{t-s}^{t+\delta-s} \int_{\{\ell\}} {\bf1}_{u \le \Upsilon^N_\ell(s^-)} \overline{Q}_{\ell'}(ds,  du, dr, d\theta) \right)^2   \right] \non  \\
& \quad + 18 L \epsilon^{-2} \E \left[ \frac{1}{N^2} \sum_{\ell'=1}^L  \left(   \int_{(t-T')^+}^{t+\delta}  
 \int_{t-s}^{t+\delta-s}p_{\ell',\ell}(r) F(dr)   \Upsilon^N_{\ell}(s) ds \right)^2   \right]  \non \\
 &=  18 \epsilon^{-2} \E \left[   \frac{1}{N} \sum_{\ell'=1}^L \int_{(t-T')^+}^{t+\delta}  
 \int_{t-s}^{t+\delta-s}p_{\ell',\ell}(r) F(dr)   \bar\Upsilon^N_\ell(s) ds  \right] \non \\
& \quad + 18 L \epsilon^{-2} \E \left[ \sum_{\ell'=1}^L  \left(   \int_{(t-T')^+}^{t+\delta}  
 \int_{t-s}^{t+\delta-s}p_{\ell',\ell}(r) F(dr)  \bar \Upsilon^N_{\ell}(s) ds \right)^2 \right]   \non\\
 & \le  18 \epsilon^{-2} \lambda^* \beta^*   \frac{1}{N} \sum_{\ell'=1}^L \int_{(t-T')^+}^{t+\delta}  
 \int_{t-s}^{t+\delta-s}p_{\ell',\ell}(r) F(dr)  ds \non\\
& \quad + 18 L \epsilon^{-2} (\lambda^* \beta^*)^2 \sum_{\ell'=1}^L  \left(   \int_{(t-T')^+}^{t+\delta}  
 \int_{t-s}^{t+\delta-s}p_{\ell',\ell}(r) F(dr) ds \right)^2 
\end{align} 
where the first term on the right hand converges to zero as $N\to \infty$. It remains to consider the second term divided by $\delta$.
Each summand in the sum over $\ell'$ is bounded from above by (with $F(s)=0$ for $s<0$)
\begin{align*}
\left(\int_0^{t+1}[F(t-s+\delta)-F(t-s)]ds\right)^2&=\left(\int_{-1}^t[F(s+\delta)-F(s)]ds\right)^2\\
&=\left(\int_0^{t+\delta}F(r)dr-\int_0^tF(s)ds\right)^2\\
&\le \delta^2\,.
\end{align*}
We have shown that this term satisfies \eqref{eqn-V-diff-c1}.
Now for the second term on the right hand side of  \eqref{eqn-V-diff-p1}, we have
\begin{align*}
&  \E\left[  \left( \frac{1}{N} \sum_{\ell'=1}^L \sum_{i=A^N_{\ell'}((t-T')^+)+1}^{A^N_{\ell'}(t+\delta)}  \int_{t-\tau_{i}^{\ell',N}}^{t+\delta-\tau_{i}^{\ell',N}} p_{\ell',\ell}(r) F(dr)  \right)^2 \right] \non \\ 
& \le  L  \E\left[  \sum_{\ell'=1}^L \left(   \int_{(t-T')^+}^{t+\delta}  
 \int_{t-s}^{t+\delta-s}p_{\ell',\ell}(r) F(dr)  d \bar{A}^N_{\ell'}(s) \right)^2 \right] \non \\ 
 & \le 2L  \E\left[  \sum_{\ell'=1}^L \left(   \int_{(t-T')^+}^{t+\delta}  
 \int_{t-s}^{t+\delta-s}p_{\ell',\ell}(r) F(dr) d \bar{M}^N_{A,\ell'}(s) \right)^2 \right] \non \\ 
 & \quad + 2L  \E\left[  \sum_{\ell'=1}^L \left(   \int_{(t-T')^+}^{t+\delta}  
 \int_{t-s}^{t+\delta-s}p_{\ell',\ell}(r) F(dr)  \bar{\Upsilon}^N_{\ell'}(s) ds \right)^2 \right]\,,
\end{align*}
where the first term converges to zero as $N\to\infty$ by the convergence $ \bar{M}^N_{A,\ell'}(s)\to 0$ in mean square, locally uniformly in $t$, and the second is estimated as the second term in \eqref{eqn-V-diff-p2}. The third term on the right hand side of  \eqref{eqn-V-diff-p1} satisfies (forgetting the sum over $\ell'$ for notational simplicity)
\begin{align*}
\P\left(\sup_{0\le t\le T}|\bar{A}^N_{\ell'}(t+\delta)-\bar{A}^N_{\ell'}(t)|>\epsilon'\right)&\le
\frac{1}{(\epsilon')^2}\E\left[\left(\sup_{0\le t\le T}\int_t^{t+\delta}\bar\Upsilon^N_{\ell'}(s)ds+2\sup_{0\le t\le T+\delta}|\bar{M}^A_{N,\ell'}(t)|\right)^2\right]
\end{align*}
It follows readily from the bound on $\bar\Upsilon^N_{\ell'}$ and the properties of the sequence of martingales $\bar{M}^A_{N,\ell'}$ that
\begin{align*}
\limsup_N\frac{1}{\delta}\P\left(\sup_{0\le t\le T}|\bar{A}^N_{\ell'}(t+\delta)-\bar{A}^N_{\ell'}(t)|>\epsilon'\right)&\le\frac{C}{\epsilon'^2}\delta\,.
\end{align*}
Combining these results gives us the property in \eqref{eqn-V-diff-c1}.

For the second condition in (ii) of Theorem \ref{thm-DD-conv0}, we show that 
for $\ep>0$, as $\delta\to 0$, 
\begin{align}\label{eqn-V-diff-c2}
\limsup_N \sup_{\mfa \in [0, T']} \frac{1}{\delta} \P \left(\sup_{v\in [0,\delta]} \sup_{t\in [0,T]} \big| V^N(t,\mfa+v) - V^N(t,\mfa)\big| >\epsilon \right) \to 0. 
\end{align} 
We have 
\begin{align*}
 & \big| V^N(t,\mfa+v) - V^N(t,\mfa)\big|  \\
 & =  \Bigg|  \frac{1}{N} \sum_{\ell'=1}^L \sum_{i=A^N_{\ell'}((t-\mfa-v)^+)+1}^{A^N_{\ell'}((t-\mfa)^+)} \bigg({\bf1}_{\tau_{i}^{\ell',N}+\eta^{\ell'}_i \le t}  {\bf1}_{X_i^{\ell'}(\eta^{\ell'}_i) = \ell} -  \int_0^{t-\tau_{i}^{\ell',N}} p_{\ell',\ell}(r) F(dr) \bigg)\Bigg| \\
 & \le  \sum_{\ell'=1}^L (\bar{A}^N_{\ell'}((t-\mfa)^+) - \bar{A}^N_{\ell'}((t-\mfa-v)^+)).
\end{align*} 
Thus
\begin{align*}
&  \P \left(\sup_{v\in [0,\delta]} \sup_{t\in [0,T]} \big| V^N(t,\mfa+v) - V^N(t,\mfa)\big| > \epsilon \right) \\
& \le \P \left(\sup_{v\in [0,\delta]} \sup_{t\in [0,T]} \sum_{\ell'=1}^L (\bar{A}^N_{\ell'}((t-\mfa)^+) - \bar{A}^N_{\ell'}((t-\mfa-v)^+)) > \varepsilon \right) \\
& \le \P \left( \sup_{t\in [0,T]} \sum_{\ell'=1}^L (\bar{A}^N_{\ell'}((t-\mfa)^+) - \bar{A}^N_{\ell'}((t-\mfa-\delta)^+)) > \varepsilon \right).
\end{align*} 
Then the claim in \eqref{eqn-V-diff-c2} follows the same argument as in the third term on the right hand side of  \eqref{eqn-V-diff-p1}. 
\end{proof} 

As an immediate consequence of Lemma \ref{lem-bar-mfI1-conv}, we obtain the following convergence results for $(\bar{R}^{N,0}_\ell(t), \bar{R}^{N,1}_\ell(t))$. 

\begin{coro}
 \label{coro-bar-mfI1-conv}
Under Assumptions \ref{AS-initial} and \ref{AS-lambda}, along a convergent subsequence of $\bar{A}^N_\ell$ with limit $\bar{A}_\ell$,  for each $\ell\in \sL$, 
$\bar{R}^{N,0}_\ell(t)\to \bar{R}^{0}_\ell(t)$ in probability, uniformly in $t$, and 
$\bar{R}^{N,1}_\ell(t)\Rightarrow \bar{R}^{1}_\ell(t)$ in $\bD$, as $N\to\infty$, where
\begin{align} 
\bar{R}^{0}_\ell(t) &=  \bar\mfI^{0}_\ell(t, \infty)=
\sum_{\ell'=1}^L  \int_0^{\infty} \bigg( \int_0^t p_{\ell',\ell}(u)   F_0(du|y) \bigg) \bar\mfI_{\ell'}(0,dy)\,,\label{eqn-bar-R0}\\
\bar{R}^{1}_\ell(t) &=  \bar\mfI^{1}_\ell(t, \infty)= \sum_{\ell'=1}^L \int_{0}^t \int_0^{t-s} p_{\ell',\ell}(u) F(du)  d \bar{A}_{\ell'}(s) \,. \label{eqn-bar-R1}
\end{align}
\end{coro}

\medskip

\begin{proof}[{\bf Proof of the convergence of $(\bar{S}^N_\ell, \bar{\mfI}^N_\ell, \bar{R}^N_\ell)_{\ell\in\sL}$.}]

We first consider the convergence provided with the convergent subsequence of $(\bar{A}^N_\ell)_{\ell\in \sL}$ with the limit
$(\bar{A}_\ell)_{\ell\in \sL}$ in Lemma \ref{lem-An-conv}. 

By \eqref{eqn-bar-Sn}, and by the convergence of  $(\bar{M}^{N}_{S,\ell, \ell'}, \ell, \ell\in \sL) \to 0$ in Lemma \ref{lem-barM-conv} and that of $( \bar{S}^N_{\ell}(0), \ell\in \sL) \to  ( \bar{S}_{\ell}(0), \ell\in \sL)$ under Assumption \ref{AS-initial},  applying the continuous mapping theorem, we obtain the  convergence of $(\bar{S}^N_\ell, \ell\in \sL)$ to $(\bar{S}_\ell, \ell\in \sL)$ in $\bD^L$ as $N\to \infty$, where 
\begin{align} \label{eqn-bar-S-1}
\bar{S}_{\ell}(t) &= \bar{S}_{\ell}(0) - \bar{A}_{\ell} (t)  + \sum_{\ell'=1}^L \bigg( \nu_{\ell',\ell}^{S} \int_0^t \bar{S}_{\ell'}(s)ds- \nu_{\ell,\ell'}^{S} \int_0^t \bar{S}_{\ell}(s)ds \bigg). 
\end{align}

We want to show the convergence of $(\bar\mfI^N_\ell(t, \mfa), \ell\in \sL)$ to $(\bar\mfI_\ell(t, \mfa), \ell\in \sL) $ locally uniformly in $t$ and $\mfa$ as $N\to \infty$, where  
\begin{align} \label{eqn-bar-mfI-1} 
\bar\mfI_\ell(t, \mfa) 
&  = \bar\mfI_{\ell}(0,(\mfa-t)^+ ) - \bar\mfI^{0}_\ell(t, \mfa) + \bar{A}_{\ell} (t)  - \bar{A}_{\ell} ((t-\mfa)^+)   
 - \bar\mfI^{1}_\ell(t, \mfa)  \non \\
& \qquad + \sum_{\ell'=1}^L \bigg( \nu_{\ell',\ell}^{I} \int_{(t-\mfa)^+}^t \bar\mfI_{\ell'}(s, \mfa-(t-s)) ds- \nu_{\ell,\ell'}^{I} \int_{(t-\mfa)^+}^t \bar\mfI_{\ell}(s, \mfa-(t-s)) ds \bigg). 
\end{align}
We first deduce from \eqref{eqn-bar-mfI-1} an explicit formula for $\bar\mfI_\ell(t, \mfa)$ in terms of 
$\mfI_{\ell}(0, \cdot)$, $ \bar\mfI^{0}_\ell(t, \mfa)$, $\bar{A}_\ell$ and $ \bar\mfI^{1}_\ell(t, \mfa)$. For that sake, 
we use again the matrix $Q$ defined at the start of section \ref{sec-PDE}.
\begin{lemma} \label{lem-mfI-map} 
The row vector $\{\bar{\mfI}_\ell(t,\mfa),\ 1\le \ell\le L,\ t\ge0, \mfa>0\}$ is given by the formula
\begin{equation}\label{mfI}
\begin{aligned}
\bar{\mfI}(t,\mfa)&=\bar\mfI(0,(\mfa-t)^+ ) - \bar\mfI^{0}(t, \mfa) + \bar{A} (t)  - \bar{A} ((t-\mfa)^+)   
 - \bar\mfI^{1}(t, \mfa)  \\
 &\quad+\int_{(t-\mfa)^+}^t \Big\{\bar\mfI(0,(\mfa-t)^+ ) - \bar\mfI^{0}(s, \mfa-t+s)\\
 &\quad\quad\quad\quad + \bar{A} (s)  - \bar{A} ((t-\mfa)^+)   
 - \bar\mfI^{1}(s, \mfa-t+s)\Big\}e^{Q(t-s)}Qds\,. 
 \end{aligned}
 \end{equation}
\end{lemma}
\begin{proof}
Equation 
\eqref{eqn-bar-mfI-1} for all $t\ge0$, $\mfa\ge0$ implies that for all $(t-\mfa)^+\le s\le t$, we have the following identity between row vectors
\begin{align*}
\bar\mfI(s, \mfa-t+s)&=\bar\mfI(0,(\mfa-t)^+ ) - \bar\mfI^{0}(s, \mfa-t+s) + \bar{A} (s)- \bar{A} ((t-\mfa)^+) -  \bar\mfI^{1}(s, \mfa-t+s) \\
&\quad  + \int_{(t-\mfa)^+}^s \bar\mfI(r,\mfa-t+r) Q dr\,.
\end{align*}
It follows that $\bar\mfI(t, \mfa) $ is the value at time $s=t$ of the solution to the system of linear ODEs: 
\begin{align}\label{odex}
x(s)=f(s)+\int_{(t-\mfa)^+}^sx(r)Qdr,
\end{align}
where, for $1\le \ell\le L$,
\[ f_\ell(s)=\bar\mfI_{\ell}(0,(\mfa-t)^+ ) - \bar\mfI^{0}_\ell(s, \mfa-t+s) + \bar{A}_{\ell} (s)- \bar{A}_{\ell} ((t-\mfa)^+) -  \bar\mfI^{1}_\ell(s, \mfa-t+s) \,.\]
Formula \eqref{mfI} now follows readily from the explicit formula for the solution of the linear ODE \eqref{odex}.
\end{proof}
Comparing \eqref{eqn-bar-mfI-N} and \eqref{eqn-bar-mfI-1}, we deduce that the row vector $\bar{\mfI}^N(t,\mfa)$ is given by an analog of formula \eqref{mfI}, namely
\begin{equation}\label{mfIN}
\begin{aligned}
\bar{\mfI}^N(t,\mfa)&=\bar\mfI^N(0,(\mfa-t)^+ ) - \bar\mfI^{N,0}(t, \mfa) + \bar{A}^N (t)  - \bar{A}^N ((t-\mfa)^+)   
 - \bar\mfI^{N,1}(t, \mfa) + \bar{\mathcal{M}}^N(t,\mfa) \\
 &\quad+\int_{(t-\mfa)^+}^t \Big\{\bar\mfI^N(0,(\mfa-t)^+ ) - \bar\mfI^{N,0}(s, \mfa-t+s) + \bar{A}^N (s)  - \bar{A}^N ((t-\mfa)^+)   \\&\quad\quad\quad\quad\quad\quad
 - \bar\mfI^{N,1}(s, \mfa-t+s)+ \bar{\mathcal{M}}^N(s,\mfa-t+s)\Big\}Qe^{Q(t-s)}ds,
 \end{aligned}
\end{equation}
where
\[\bar{\mathcal{M}}^N_\ell(t,\mfa)=\sum_{\ell'=1}^L\Big(\bar{M}^{N}_{\mfI, \ell', \ell} (t,\mfa) - \bar{M}^{N}_{\mfI,\ell, \ell'} (t,\mfa)\Big)\,.\]
Comparing \eqref{mfIN} and \eqref{mfI}, it now follows from Assumption \ref{AS-initial}, Lemma \ref{lem-An-conv}, Lemma \ref{lem-barM-conv}, 
Lemma \ref{lem-bar-mfI0-conv} and Lemma \ref{lem-bar-mfI1-conv} that 
$\bar{\mfI}^N(t,\mfa)\Rightarrow\bar{\mfI}(t,\mfa)$
for the topology of locally uniform convergence in $t$ and $\mfa$.

As a consequence, letting $\bar{I}_\ell(t) = \bar\mfI_\ell(t, \infty)$, we also get the weak convergence of $(\bar{I}^N_\ell, \ell\in \sL)$ to $(\bar{I}_\ell, \ell\in \sL) $  locally uniformly in $t$ as $N\to \infty$, where 
\begin{align}\label{eqn-bar-I-1}
\bar{I}_\ell(t) &= \bar{I}_{\ell}(0) + \bar{A}_{\ell} (t)  -\bar{R}^0_\ell(t)  - \bar{R}^1_\ell(t) 
+  \sum_{\ell'=1}^L \int_0^t \Big( \nu^I_{\ell', \ell}\bar{I}_{\ell'}(t) - \nu^I_{\ell, \ell'}\bar{I}_{\ell}(t)  \Big) ds\,.       
\end{align} 
Then by \eqref{eqn-bar-Rn}, and by the convergence of $(\bar{M}^{N}_{R,\ell, \ell'}, \ell, \ell'\in \sL) \to 0$ in Lemma \ref{lem-barM-conv}, of  $(\bar{R}^{N,0}_\ell, \bar{R}^{N,1}_\ell, \ell\in \sL) \to (\bar{R}^{0}_\ell, \bar{R}^{1}_\ell, \ell\in \sL)$  in  Corollary \ref{coro-bar-mfI1-conv}, 
and that of $( \bar{R}^N_{\ell}(0), \ell\in \sL) \to  ( \bar{R}_{\ell}(0), \ell\in \sL)$ under Assumption \ref{AS-initial},  applying the continuous mapping theorem, we obtain the convergence of $(\bar{R}^N_\ell, \ell\in \sL)$ to $(\bar{R}^N_\ell, \ell\in \sL)$ in $\bD^L$ as $N\to \infty$, where 
\begin{align} \label{eqn-bar-R-1} 
\bar{R}_\ell(t)
& =  \bar{R}_\ell(0)+  \bar{R}^{0}_\ell(t) + \bar{R}^{1}_\ell(t)  + \sum_{\ell'=1}^L \bigg( \nu_{\ell',\ell}^{R} \int_0^t \bar{R}_{\ell'}(s)ds- \nu_{\ell,\ell'}^{R} \int_0^t \bar{R}_{\ell}(s)ds \bigg)\,. 
\end{align} 

Next we identify the limit  $(\bar{A}_\ell, \ell \in \sL)$ in terms of the limits $(\bar{\mfF}_\ell, \bar{S}_\ell, \bar{I}_\ell, \bar{R}_\ell, \ell \in \sL)$ and let $\bar{B}_\ell= \bar{S}_\ell+ \bar{I}_\ell+ \bar{R}_\ell$.
Recall that we have shown in the proof of Lemma \ref{uniqBC} that for each $0 \le t \le T$, 
$1\le\ell\le L$, $\bar{S}_\ell(t)+ \bar{I}_\ell(t)+ \bar{R}_\ell(t) \ge c_T$. 
The mapping from $(\bar{S}_\ell(t), \bar{I}_\ell(t), \bar{R}_\ell(t),  \sum_{\ell'=1}^L \beta_{\ell',\ell} \bar{\mfF}_{\ell'})$ to $\bar\Upsilon_\ell(t)$ is continuous in the Skorohod topology whenever $\bar{S}_\ell(t)+ \bar{I}_\ell(t)+ \bar{R}_\ell(t)>0$. Then we obtain the convergence 
\[
\bar\Upsilon^N_\ell(t) = \frac{\bar{S}^N_\ell(t) \sum_{\ell'=1}^L \beta_{\ell',\ell} \bar{\mfF}^N_{\ell'}}{(\bar{S}^N_\ell(t)+ \bar{I}^N_\ell(t)+ \bar{R}^N_\ell(t))^\gamma}  \Rightarrow \bar\Upsilon_\ell(t) :=\frac{\bar{S}_\ell(t) \sum_{\ell'=1}^L \beta_{\ell',\ell} \bar{\mfF}_{\ell'}}{(\bar{S}_\ell(t)+ \bar{I}_\ell(t)+ \bar{R}_\ell(t))^\gamma} \,,
\]
in $\bD$ as $N\to\infty$. Then by Lemma \ref{lem-An-conv}, we obtain the convergence of $(\bar{A}^N_\ell, \ell \in \sL)$
to $(\bar{A}_\ell, \ell \in \sL)$ in $\bD^L$, where 
\[
\bar{A}_\ell(t) = \int_0^t \bar\Upsilon_\ell(s) ds\,,
\]
with $\bar\Upsilon_\ell(s) $ given above. Since all converging sub-sequences have the same limit, which is deterministic, we have the convergence in probability of the whole sequence.
This completes the proof. 
\end{proof}

\medskip

\section{Appendix}

The following theorem was stated in Theorem 5.1 in \cite{PP-2021}. It extends the Corollary on page 83 of \cite{billingsley1999convergence}, and also   Theorem 3.5.1 in Chapter 6 of \cite{khoshnevisan2002multiparameter}
in  the space $\bC([0,1]^k, \RR)$. 

\begin{theorem} \label{thm-DD-conv0}
Let $\{X^N: N \ge 1\}$ be a sequence of random elements in $\bD_\bD$.
If the following two conditions are satisfied: for any $T,S>0$, 
\begin{itemize}
\item[(i)] for any $\ep>0$, $  \sup_{t \in [0,T]}\sup_{s\in [0,S]} \P \big(  |X^N(t, s)|> \ep \big) \to 0$ as $N\to\infty$, and
\item[(ii)] for any $\ep>0$, as $\delta\to0$,
\begin{align*} 
& \limsup_{N\to\infty}  \sup_{t\in [0,T]} \frac{1}{\delta}  \P \bigg(  \sup_{u \in [0,\delta]}\sup_{s \in [0,S]} |X^N(t+u,s) - X^N(t,s)| > \ep\bigg) \to 0, \\
& \limsup_{N\to\infty}  \sup_{s\in [0,S]} \frac{1}{\delta}  \P \bigg(  \sup_{v \in [0,\delta]}\sup_{t \in [0,T]} |X^N(t,s+v) - X^N(t,s)| > \ep\bigg) \to 0, 
\end{align*}
\end{itemize}
then  $X^N(t,s)\to 0 $ in probability, locally uniformly in $t$ and $s$, as $N\to\infty$. 
\end{theorem}

The following lemma was stated in Lemma 5.1 in \cite{PP-2021}. The spaces $\bD_\uparrow$ and $\bC_\uparrow$ are the subspaces of $\bD$ and $\bC$ of increasing functions. 
\begin{lemma}\label{le:Portmanteau}
Let $f\in \bD$ and $\{g_N\}_{N\ge1}$ be a sequence of elements of $\bD_\uparrow$ which is such that
$g_N\to g$ locally uniformly, where $g\in \bC_\uparrow$. Then for any $T>0$,
\[ \int_{[0,T]} f(t) g_N(dt)\to \int_{[0,T]} f(t)g(dt)\, .\]
\end{lemma}

\bigskip

\section*{Acknowledgement}
The authors warmly thank the anonymous Referee for correcting an error in an earlier version and for many helpful comments that improved the exposition of the paper. G. Pang is partly supported by the US National Science of Foundation grant DMS-2216765.

\bibliographystyle{abbrv}
\bibliography{Epidemic-Age-PDE}

\end{document}